\documentclass[reqno]{amsart}
\usepackage{amsfonts}
\usepackage{mathrsfs}
\usepackage{amsmath}
\usepackage{amssymb}
\usepackage{esint}
\usepackage[numbers,sort&compress]{natbib}
\usepackage{enumitem}
\usepackage{tikz}
\usetikzlibrary{arrows}
\usepackage{xcolor}
\usepackage{subcaption}
\usepackage{cases}
\AtBeginDocument{}
\allowdisplaybreaks
\definecolor{naturegreen}{RGB}{123, 173, 115}
\definecolor{naturepurple}{RGB}{221,160,221}
\begin{document}
\title[]
{Boundedness and blow-up for a quasilinear	Keller-Segel system with flux limitation \\and indirect signal production}

\author[Y. Zhou and C. Liu]
{Yue Zhou and Changchun Liu$^*$ }

\address{Yue Zhou\hfill\break School of Mathematics, Jilin
University, Changchun 130012, China}
\email{2747376759@qq.com}
\address{Changchun Liu\hfill\break School of Mathematics, Jilin
University, Changchun 130012, China}
\email{liucc@jlu.edu.cn}
\date{}
%\thanks{This work is supported by the Science and Technology Development Plan
%	Project of Jilin Province, China (No. 20250102009JC)}
\thanks{$^*$Corresponding author. E-mail: liucc@jlu.edu.cn}

\begin{abstract}
The quasilinear	Keller-Segel system with flux limitation and indirect signal production
\begin{equation*}
\begin{cases}
u_t=\nabla\cdot\left(D(u)\nabla u\right)
-\nabla\cdot\left(u(1+\left|\nabla v\right|^2)^\sigma\nabla v\right),
&x\in\Omega,~t>0,
\\
0=\Delta v-v+w,&x\in\Omega,~t>0,
\\
w_t=-w+u,&x\in\Omega,~t>0,
\end{cases}
\end{equation*}
under homogeneous Neumann boundary conditions in a smooth bounded domain $\Omega\subset\mathbb{R}^N$ is considered, where $D(u)\simeq u^{m-1}$ as $u\simeq\infty$.
We conclude that
\begin{itemize}[leftmargin=*]
\item For $N=1$ and any $\sigma\in\mathbb{R}$, if $m\geq0$, the classical solution exists globally, and it is moreover bounded if $m>0$. However, if $m<0$ and $\Omega$ is a ball, there exist radially symmetric initial data such that the classical solution exhibits finite-time blow-up.
\item For any $N\geq2$ and $m>1-\frac{1}{N}$, if $\sigma\leq \frac{mN+2-2N}{2N-2}$, the classical solution is global. Furthermore, if $\sigma<\frac{mN+2-2N}{2N-2}$, the corresponding solution is uniformly bounded.
\item For any $N\geq2$ and $m<2-\frac{2}{N}$, if $\sigma>\max\left\{\frac{N}{2-2N},\frac{mN+2-2N}{2N-2}\right\}$ and $\Omega$ is a ball, there exist radially symmetric initial data such that the classical solution blows up in finite time.
\end{itemize}
\end{abstract}

\maketitle
\let\oldsection\section
\renewcommand\section{\setcounter{equation}{0}\oldsection}
\renewcommand\thesection{\arabic{section}}
\renewcommand\theequation{\thesection.\arabic{equation}}
\newtheorem{theorem}{\indent Theorem}[section]
\newtheorem{lemma}{\indent Lemma}[section]
\newtheorem{proposition}{\indent Proposition}[section]
\newtheorem{definition}{\indent Definition}[section]
\newtheorem{remark}{\indent Remark}[section]
\newtheorem{corollary}{\indent Corollary}[section]
\section{Introduction}
Chemotactic movement is a key driving mechanism underlying collective behaviors in numerous domains, spanning from immunology, carcinogenesis and disease evolution to ecology and social systems \cite{P}. As the prototypical model for populations attracted by endogenously produced chemical signals, the cross-diffusion system proposed by Keller and Segel \cite{KS} reads
\begin{equation}\label{ks}
\begin{cases}
u_t=\Delta u-\nabla\cdot(u\nabla v),
\\
v_t=\Delta v-v+u,
\end{cases}
\end{equation}
where $u$ and $v$ are the population density and signal concentration, respectively. Both existence and blow-up phenomena for system \eqref{ks} and associated models have been thoroughly studied \cite{CS,CG,W100}.

The gradient of $v$ appears linearly in \eqref{ks}. In contrast, to characterize chemotaxis phenomena in the presence of large signal gradients, the following system is considered:
\begin{equation}\label{1-5}
\begin{cases}
u_t=\Delta u
-\nabla\cdot\left(u(1+|\nabla v|^2)^\sigma\nabla v\right),
&x\in\Omega,~t>0,
\\
0=\Delta v-\mu(t)+u,
\quad\mu(t)=\frac{1}{\left|\Omega\right|}\int_\Omega u(\cdot,t),
&x\in\Omega,~t>0,
\end{cases}
\end{equation}
which is supplemented with homogeneous Neumann boundary conditions. Results from Winkler \cite{W6}, Kohatsu \cite{K2}, and Tello \cite{T} showed that in $N=1$, problem \eqref{1-5} possesses a globally bounded classical solution for any $\sigma\in\mathbb{R}$, provided that the initial data are nonnegative and continuous. Meanwhile, for $N\geq2$, $\sigma=-\frac{N-2}{2(N-1)}$ constitutes the critical blow-up exponent for problem \eqref{1-5}. Recently, Zhang et al. \cite{ZMT} considered the quasilinear case of system \eqref{1-5}, which replaces $\Delta u$ with $\nabla\cdot((u+1)^{m-1}\nabla u)$. The authors proved that boundedness holds if $N\geq 1,~\sigma<\frac{mN+2-2N}{2N-2}$ and $m>1-\frac{1}{N}$, and also in the critical case $\sigma=\frac{mN+2-2N}{2N-2}$ with $m\geq 1,~N\geq 2$ under a small-mass condition. Conversely, for radial solutions in a ball with $N\geq 3$, finite-time blow-up occurs in the critical case with large mass or in the regime $0>\sigma>\min\{\frac{mN+2-2N}{2N-2},-\frac{N-2}{2N-2}\}$. It was shown in \cite{ZMT} that these results are optimal for $m \geq 1$, while the case $0 < m < 1$ remains open. The authors further noted that their method does not extend to the two-dimensional case. For the system
\begin{equation}\label{1-6}
\begin{cases}
u_t=\nabla\cdot\left((u+1)^{m-1}\nabla u
-u(1+|\nabla v|^2)^\sigma\nabla v\right),
&x\in\Omega,~t>0,
\\
0=\Delta v-v+u,
&x\in\Omega,~t>0,
\end{cases}
\end{equation}
under the Neumann boundary conditions, global boundedness results have been obtained in \cite{W7,ZY} for $m\geq1$ and $\sigma<\min\{\frac{mN+2-2N}{2N-2},0\}$. \textbf{However, the behavior of solutions in the complementary parameter regimes has not been studied yet and is expected to exhibit more complex dynamics including blow-up}.

It is worth noting that all the aforementioned models \eqref{ks}-\eqref{1-6} feature direct signal production, where the chemical substance is assumed to be secreted by the cells themselves. In many physiological and ecological settings, however, the emission of chemoattractants involves a more intricate, multi-stage mechanism. In parallel, as a biologically more realistic modeling framework, chemotaxis systems with indirect signal production have been intensively investigated. A prototypical example is given by the following system
\begin{equation}\label{1-3}
\begin{cases}
u_t=\Delta u-\nabla\cdot\left(u\nabla v\right),
&x\in\Omega,~t>0,
\\
0=\Delta v-\mu(t)+w,
\quad
\mu(t)=\frac{1}{\left|\Omega\right|}\int_\Omega w(\cdot,t),
&x\in\Omega,~t>0,
\\
w_t=-w+u,
&x\in\Omega,~t>0,
\end{cases}
\end{equation}
which is a simplified version of the chemotaxis model originally proposed by Strohm et al. \cite{STP}. Here, $u$ and $w$ represent the densities of flying and nesting mountain pine beetles, respectively, whereas $v$ denotes the concentration of beetle pheromones. The mathematical properties of \eqref{1-3} and its quasilinear variants, particularly regarding the criticality of global existence versus blow-up, have been intensively investigated \cite{TW2,L2,JL2,FLT}.
To understand how indirect signaling interacts with gradient-dependent sensitivity limits, Jin et al. \cite{JLY} recently studied
\begin{equation}\label{1-7}
\begin{cases}
u_t=\Delta u
-\nabla\cdot\left(u(1+\left|\nabla v\right|^2)^\sigma\nabla v\right),
\;\;&\hbox{}\;x\in\Omega,~t>0,
\\
0=\Delta v-\mu(t)+w,
\quad\mu(t)=\frac{1}{\left|\Omega\right|}\int_\Omega u(\cdot,t),
\;\;&\hbox{}\;x\in\Omega,~t>0,
\\
w_t=-w+u,
\;\;&\hbox{}\;x\in\Omega,~t>0.
\end{cases}
\end{equation}
For the Neumann initial-boundary value problem of \eqref{1-7}, they showed that if either $N=1$ and $\sigma\in\mathbb{R}$, or $N\geq2$ and $\sigma<-\frac{N-2}{2N-2}$, then for any properly regular initial data, the corresponding solution exists globally and remains bounded. Conversely, if $\Omega$ is a ball and $0\geq\sigma>-\frac{N-2}{2N-2}$, then there exist radially symmetric initial data such that the corresponding solution blows up in finite time. From these conclusions, it is apparent that the results in \cite{JLY} for $N \geq 2$ are restricted to $\sigma \leq 0$, \textbf{leaving the case $\sigma>0$ completely open. Furthermore, the case where the term $-\mu(t)$ in the equation for $v$ is replaced by $-v$ in \eqref{1-7} remains a gap in the current results}.

Since $-\mu(t)$ depends only on the time variable $t$, while $v$ depends on both the spatial variable $x$ and time $t$, the latter setting is more involved. As a result, the techniques employed in the direct signaling model \cite{ZMT} and the indirect signaling model \cite{JLY} are no longer directly applicable. Meanwhile, the analysis of the indirect signaling model is inherently more complex than that of the direct signaling model \eqref{1-6}, as it involves two evolving components $u$ and $w$ coupled through the elliptic equation for $v$, requiring more delicate multi-variable energy estimates rather than a single energy functional for $u$. Furthermore, it is well known that the competition between chemotactic aggregation and diffusion determines whether solutions exist globally or blow up. Therefore, we investigate the following quasilinear Keller-Segel system with flux limitation and indirect signal production
\begin{equation}\label{wen}
\begin{cases}
u_t=\nabla\cdot\left(D(u)\nabla u\right)
-\nabla\cdot\left(u(1+\left|\nabla v\right|^2)^\sigma\nabla v\right),
\;\;&\hbox{}\;x\in\Omega,~t>0,
\\
0=\Delta v-v+w,
\;\;&\hbox{}\;x\in\Omega,~t>0,
\\
w_t=-w+u,
\;\;&\hbox{}\;x\in\Omega,~t>0,
\\
\frac{\partial u}{\partial\nu}
=\frac{\partial v}{\partial\nu}=0,
\;\;&\hbox{}\;x\in\partial\Omega,~t>0,
\\
u(x,0)=u_{0},~w(x,0)=w_{0},
\;\;&\hbox{}\;x\in\Omega,
\end{cases}
\end{equation}
where
\begin{equation}\label{D}
D>0\text{~on~}[0,\infty),
\end{equation}
satisfying $D(\xi)\simeq\xi^{m-1}$ as $u\simeq\infty$. We aim to determine the critical blow-up exponent for problem \eqref{wen}.

\textbf{Main difficulties.} The system \eqref{wen} is challenging due to the presence of nonlinearly coupled terms and the spatiotemporal dependence of $v$.
\begin{itemize}
\item[1.]
Compared with the case where the diffusion and chemotaxis terms depend on the population density $u$ itself (see, e.g., \cite{TW}), the inclusion of the nonlinear flux-limited term $-\nabla\cdot(u(1+|\nabla v|^2)^\sigma\nabla v)$ in \eqref{wen} renders the analysis more complicated. In particular, the primary difficulty lies in establishing the comparison principle for the proof of blow-up. This explains why the analytical arguments in \cite{JLY,ZMT} were restricted to the regime $\sigma\leq 0$. Meanwhile, \cite{ZMT,TW,JLY} reveal that the spatial average $-\mu(t)$ can be eliminated during the proof of the comparison principle. In contrast, our setting involving the linear decay $-v$ inherently introduces a fully coupled triple of state variables $(u, v, w)$ into the differential inequality of the comparison principle. This extra layer of coupling prevents the elimination of the signal term.
\item[2.]
The introduction of nonlinear diffusion into the framework of \cite{JLY} complicates the identification of the critical blow-up exponent. Moreover, the classical technique based on the Neumann heat semigroup employed in the linear case to prove the existence of solutions is no longer applicable.
\end{itemize}

\textbf{Our strategies and novelties.} To overcome the aforementioned obstacles, this paper introduces the following approaches.

\begin{itemize}
\item[1.]
For radial solutions, we establish a precise pointwise upper bound for $|v'(r)|$ with respect to the radius $r$ (Lemma \ref{le2.3}). This key gradient estimate successfully unlocks the comparison principle for the parameter regime $\sigma > 0$, thereby breaking the artificial boundary encountered in \cite{JLY,ZMT}. Furthermore, when verifying the differential inequality $\mathcal{P}[\underline{U},\underline{W}](s,t)\leq0$, the core strategy is to leverage the chemotactic term to dominate the diffusion effect. This domination crucially relies on the lower bound estimates for $\widehat{W}-eb(s)$ (Lemma \ref{le3.5}). This requirement motivates us to construct a new subsolution that is distinct from the one employed in \cite{ZMT, TW}.
\item[2.]
To overcome the failure of the Neumann heat semigroup technique, we derive sharper energy estimates, which enable us to cover a broader parameter regime and establish the critical exponent. Notably, our approach is equally applicable to the direct signaling model \eqref{1-6}, yielding a wider parameter regime than those previously obtained in \cite{W7,ZY}.
\end{itemize}

We prove that for any $1-\frac{1}{N} < m < 2-\frac{2}{N}~(N \geq 2)$, $\frac{mN+2-2N}{2N-2}$ is the critical blow-up exponent for problem \eqref{wen}. This paper is organized as follows. Chapter 2 is devoted to presenting and proving a series of technical lemmas required for the subsequent analysis. The finite-time blow-up results are established in Chapter 3, where the regimes $\sigma\geq 0$ (Subsection \ref{sub3.3.1}) and $\sigma<0$ (Subsection \ref{sub3.3.2}) are analyzed separately through the construction of concise yet practical subsolutions (Lemma \ref{le3.3} and Lemma \ref{le3.4}).
In Chapter 4, we prove the existence of solutions separately for the cases $N=1$ (Section \ref{se4.1}) and $N\geq2$ (Section \ref{se4.2}).

Specifically, the main results of this paper are as follows.

\begin{theorem}\label{th1.1}
Let $\Omega=  B_R(0)\subset\mathbb{R}^N~(N\geq 1)$ with some $R>0$. Suppose that $D\in C^3\left(\left[0,\infty\right)\right)$ satisfies \eqref{D} and
\begin{equation}\label{1-8}
D(\xi)\leq K_D\xi^{m-1},
\quad\text{for all~}
\xi>\xi_0,
\end{equation}
where $\xi_0>0,~K_D>0$ are constants, and exponents $m\in\mathbb{R},~\sigma\in\mathbb{R}$ satisfy
\begin{equation}\label{1-9}
m<2-\frac{2}{N},
\end{equation}
and
\begin{equation}\label{1-10}
\sigma\in
\begin{cases}
\left(-\infty,+\infty\right),
&\text{if}~N=1,
\\
\left(\max\{\frac{N}{2-2N},\frac{mN+2-2N}{2N-2}\},+\infty\right),
&\text{if}~N\geq 2.
\end{cases}
\end{equation}
Then for any $T^*>0$ and $M^*>M_*>0$, there exist $M^{(u)},~M^{(w)}\in C^0\left(\left[0,R\right]\right)$ such that if the initial data
\begin{equation}\label{1-11}
(u_0,w_0)\in\left(C^1(\overline{\Omega})\right)^2
\text{~are nonnegative and radially symmetric},
\end{equation}
and further satisfy
\begin{equation}\label{1-12}
M_*\leq\int_\Omega u_0\mathrm{d}x\leq M^*
\quad\text{and}\quad
M_*\leq\int_\Omega w_0\mathrm{d}x\leq M^*,
\end{equation}
together with
\begin{equation}\label{1-13}
\int_{B_r(0)}u_0\mathrm{d}x\geq M^{(u)}(r)
~\text{and}~
\int_{B_r(0)}w_0\mathrm{d}x\geq M^{(w)}(r),
\quad\text{for all}~
r\in(0,R),
\end{equation}
then the classical solution $(u,v,w)$ to the system \eqref{wen} blows up before time $T^*$. Precisely, there exists a maximal existence time $T_{\max}\in(0,T^*)$ such that
\begin{equation}\label{1-14}
\limsup_{t\nearrow T_{\max}}\|u(\cdot,t)\|_{L^\infty(\Omega)}
=\infty.
\end{equation}
\end{theorem}

\begin{theorem}\label{th1.2}
Let $\Omega\subset\mathbb{R}^N~(N\geq1)$ denote a bounded domain with smooth boundary. Suppose that $D\in C^2([0,\infty))$ satisfies \eqref{D} and
\begin{equation}\label{1-15}
D(\xi)\geq k_D\xi^{m-1},
\quad\text{for all~}
\xi>\xi_0,
\end{equation}
where $\xi_0>0,~k_D>0$ are constants, and exponents $m\in\mathbb{R},~\sigma\in\mathbb{R}$  fulfill
\begin{itemize}
\item~$N=1$.
\end{itemize}
\begin{equation*}
m\geq0
\quad\text{and}\quad
\sigma\in\mathbb{R}.
\end{equation*}
\begin{itemize}
\item~$N\geq 2$.
\end{itemize}
\begin{equation*}
m>1-\frac{1}{N}
\quad\text{and}\quad
\sigma\leq\frac{mN+2-2N}{2N-2}.
\end{equation*}
Then for every pair of nonnegative initial data
\begin{equation}\label{1-16}
(u_0, w_0)\in\left(C^1(\overline{\Omega})\right)^2,
\end{equation}
the classical solution $(u,v,w)$ of \eqref{wen} exists globally. Furthermore, if
\begin{itemize}
\item~$N=1$.
\end{itemize}
\begin{equation*}
m>0
\quad\text{and}\quad
\sigma\in\mathbb{R},
\end{equation*}
\begin{itemize}
\item~$N\geq 2$.
\end{itemize}
\begin{equation*}
m>1-\frac{1}{N}
\quad\text{and}\quad
\sigma<\frac{mN+2-2N}{2N-2},
\end{equation*}
hold, then the global solution is uniformly bounded.
\end{theorem}

\begin{remark}\label{remark 1.1}
{\rm For $N=1$, we establish that $m=0$ is the critical exponent for blow-up in \eqref{wen}. That is, for any $\sigma\in\mathbb{R}$, global existence of the classical solution of \eqref{wen} holds when $m\geq0$. In contrast, if $m<0$, there exists radially symmetric solution for which finite-time blow-up occurs.}
\end{remark}

\begin{remark}\label{remark 1.2}
{\rm For any $1-\frac{1}{N}<m<2-\frac{2}{N}~(N\geq 2)$, $\sigma=\frac{mN+2-2N}{2N-2}$ is the critical exponent for blow-up in \eqref{wen}. More precisely, global existence of the solution holds when $\sigma\leq\frac{mN+2-2N}{2N-2}$, while finite-time blow-up occurs when $\sigma>\frac{mN+2-2N}{2N-2}$. The method developed in this paper can be readily applied to model \eqref{1-7}, requiring only the replacement of $b(s)$ in \eqref{3-11} by $\frac{\mu}{N}s$. For $m=1$, our results contain and extend those in \cite{JLY} concerning existence and blow-up of solutions.}
\end{remark}

The main results of this work are captured by the following figure.

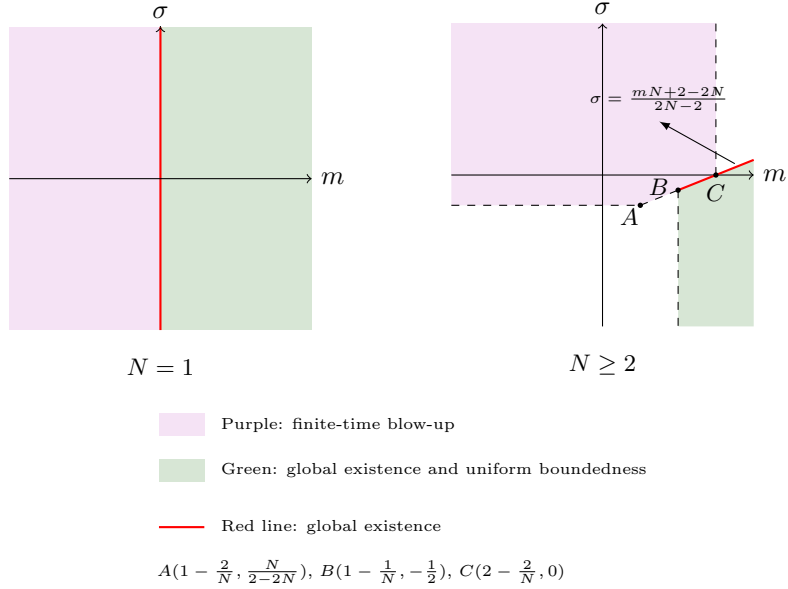
\begin{figure}[!htbp]
\centering
\begin{subfigure}[b]{0.42\textwidth}
\centering
\begin{tikzpicture}[scale=0.5]
\pgfmathsetmacro{\twoOverN}{2/1}
\pgfmathsetmacro{\xIntercept}{1 - \twoOverN}
\fill[naturepurple!30] (-4, 4) -- (0, 4) -- (0, -4) -- (-4, -4) -- cycle;
\fill[naturegreen!30] (0, 4) -- (4, 4) -- (4, -4) -- (0, -4) -- cycle;
\draw[thick, red] (0, -4) -- (0, 4);
\draw[->, thin, black] (0, 4) -- (0, 4);
\node[above, black] at (0, 4) {$\sigma$};
\draw[->, thin, black] (-4, 0) -- (4, 0) node[right] {$m$};
\node[below, black, font=\small] at (0, -4.5) {$N=1$};
\end{tikzpicture}
\end{subfigure}
\hspace{0.3cm}
\begin{subfigure}[b]{0.42\textwidth}
\centering
\begin{tikzpicture}[scale=0.5]
\fill[naturepurple!30] (-4, 0) -- (3, 0) -- (3, 4) -- (-4, 4) -- cycle;
\fill[naturepurple!30] (-4, -0.8) -- (1, -0.8) -- (3, 0) -- (-4, 0)-- cycle;
\draw[black, dashed, thin] (3, 0) -- (3, 4);
\draw[black, dashed, thin] (1, -0.8) -- (2, -0.4);
\draw[black, dashed, thin] (-4, -0.8) -- (1, -0.8);
\fill[naturegreen!30] (4, 0.4) -- (2, -0.4)-- (2, -4) -- (4, -4)-- cycle;
\draw[->, thin, black] (0, -4) -- (0, 4);
\node[above, black] at (0, 4) {$\sigma$};
\draw[->, thin, black] (-4, 0) -- (4, 0);
\node[right, black] at (4, 0) {$m$};
\draw[red, thick] (2, -0.4) -- (4, 0.4);
\draw[black, dashed, thin] (2, -0.4) -- (2, -4);
\node[black, font=\tiny, align=center] (formula) at (1.5, 2) {$\sigma = \frac{mN+2-2N}{2N-2}$};
\draw[black, -latex] (3.5, 0.3) -- (formula.south);
\fill[black] (1, -0.8) circle (2pt);
\node[left, font=\small] at (1.2, -1.1) {$A$};
\fill[black] (2, -0.4) circle (2pt);
\node[left, font=\small] at (2, -0.3) {$B$};
\fill[black] (3, 0) circle (2pt);
\node[below, font=\small] at (3, 0) {$C$};
\node[below, black, font=\small] at (0, -4.5) {$N\geq2$};
\end{tikzpicture}
\end{subfigure}

\vspace{0.3cm}
\begin{minipage}[c]{0.87\textwidth}
 \centering
\begin{tikzpicture}[font=\tiny, baseline, anchor=west]
\fill[naturepurple!30] (0,0) rectangle (0.6,0.3);
\node[right, align=left, anchor=west] at (0.7,0.15) {Purple: finite-time blow-up};
\fill[naturegreen!30] (0,-0.6) rectangle (0.6,-0.3);
\node[right, anchor=west] at (0.7,-0.45) {Green: global existence and uniform boundedness};
\draw[red, thick] (0,-1.2) -- (0.6,-1.2);
\node[right, anchor=west] at (0.7,-1.2) {Red line: global existence};
\node[font=\tiny, anchor=west] at (-0.15,-1.8) {$A(1-\frac{2}{N},\frac{N}{2-2N})$,~$B(1-\frac{1}{N},-\frac{1}{2})$,~$C(2-\frac{2}{N},0)$};
\end{tikzpicture}
\end{minipage}
\end{figure}

\section{Local existence and preliminaries}
In what follows, we simplify $\int_\Omega f(\cdot,t)\mathrm{d}x$ to $\int_\Omega f$ and use $\left\|\cdot\right\|_{L^p}=\left\|\cdot\right\|_{L^p(\Omega)}$. Additionally, the notation $c_i$ $(i=1,~2,~\dots)$ will be used to represent unspecified positive constants.

\begin{lemma}{\rm(\cite{TW2})}\label{le2.1}
Assume that $\Omega\subset\mathbb{R}^N~(N\geq 1)$ is a bounded smooth domain and that $D\in C^2([0,\infty))$ satisfies \eqref{D}. Let $(u_0, w_0)\in\left(C^1(\overline{\Omega})\right)^2$ be nonnegative. Then there exist $T_{\max}\in (0,\infty]$ and a unique triplet $(u,v,w)$ of nonnegative functions such that
\begin{equation}\label{2-1}
\begin{cases}
u\in C^0\left(\overline{\Omega}\times[0,T_{\max})\right)
\cap
C^{2,1}\left(\overline{\Omega}\times(0,T_{\max})\right),
\\
v\in C^{2,0}\left(\overline{\Omega}\times [0,T_{\max})\right),
\\
w\in C^{0,1}\left(\overline{\Omega}\times[0, T_{\max})\right),
\end{cases}
\end{equation}
which forms a classical solution of \eqref{wen} in $\Omega\times(0,T_{\max})$. Furthermore, the following properties are valid:
\begin{itemize}
\item If $T_{\max}<\infty$, then
\begin{equation}\label{2-2}
\limsup_{t\nearrow T_{\max}}\|u(\cdot,t)\|_{L^\infty(\Omega)}
=\infty.
\end{equation}
\item For all $t\in(0,T_{\max})$, we have
\begin{equation}\label{2-3}
\int_\Omega u(\cdot,t)\mathrm{d}x
=\int_\Omega u_0\mathrm{d}x.
\end{equation}
\item If $u_0$ and $w_0$ are radially symmetric, then the solution is radially symmetric.
\end{itemize}
\end{lemma}

\begin{lemma}{\rm(\cite{ZK})}\label{le2.2}
Let $T>0$, $\tau\in(0,T)$, $A>0$, $\alpha>0$ and $B>0$. Suppose that $y:~[0,T)\rightarrow [0,\infty)$ is absolutely continuous and satisfies
\begin{equation*}
y'(t)+Ay^\alpha(t)\leq h(t),
\quad\text{for a.e.}~t\in(0,T),
\end{equation*}
with some nonnegative function $h\in L_{loc}^1\left([0,T)\right)$. If
\begin{equation*}
\int_t^{t+\tau}h(s)\mathrm{d}s
\leq B,
\quad\text{for all}~t\in(0,T-\tau),
\end{equation*}
then there exists a constant $\tilde{C}=\max\left\{y_0+B,\frac{1}{\tau^{\frac{1}{\alpha}}}(\frac{B}{A})^{\frac{1}{\alpha}}+2B\right\}$,
such that
$$y(t)\leq\tilde{C},
\quad\text{for~all}~t\in(0,T).$$
\end{lemma}

\begin{lemma}\label{le2.3}
Let $\Omega=B_R(0)\subset\mathbb{R}^N (N\geq2)$ and $R>0$. Suppose that $h\in L^1(\Omega)$ is a radially symmetric function. Let $z$ be the unique radially symmetric solution to
\begin{equation}\label{2-4}
-\Delta z+z=h,
\quad\text{in~}\Omega,
\end{equation}
under homogeneous Neumann boundary conditions. If $h\geq 0$ in $\Omega$, then there exist constants $C_{i}>0 (i=1,2,3,4)$, depending only on $N$ and $R$, such that for any $r\in(0,R]$,
\begin{itemize}
\item~$N=2$.
\end{itemize}
\begin{equation}\label{2-5}
\left|z(r)\right|\leq C_1\left\|h\right\|_{L^1(\Omega)}\left(1+\left|\ln r\right|\right)+\left|z(R)\right|,
\end{equation}
and
\begin{equation}\label{2-6}
\left|z'(r)\right|
\leq \frac{C_2}{r}\left\|h\right\|_{L^1(\Omega)}.
\end{equation}
\begin{itemize}
\item~$N\geq3$.
\end{itemize}
\begin{equation}\label{2-7}
0\leq z(r)
\leq C_3\left\|h\right\|_{L^1(\Omega)}r^{2-N},
\end{equation}
and
\begin{equation}\label{2-8}
\left|z'(r)\right|
\leq C_4\left\|h\right\|_{L^1(\Omega)}r^{1-N}.
\end{equation}
\end{lemma}

\begin{proof}
\eqref{2-7} and \eqref{2-8} follow immediately from \cite[Lemma A.1]{L4}.
As $z$ is the unique radially symmetric solution to \eqref{2-4} under homogeneous Neumann boundary conditions, we have that $z$ satisfies
\begin{equation}\label{2-9}
\left\|z\right\|_{L^1(\Omega)}=\left\|h\right\|_{L^1(\Omega)},
\end{equation}
and solves
\begin{equation}\label{2-10}
-\left(rz'(r)\right)_r=r(h-z)(r),
\quad\text{for all }r\in(0,R].
\end{equation}
Integrating \eqref{2-10} over $(0,r)$ gives
\begin{equation}\label{2-11}
z'(r)
= r^{-1}\int_0^r(z-h)(\rho)\rho\mathrm{d}\rho,
\quad\text{for all }r\in(0,R].
\end{equation}
By combining \eqref{2-9} and \eqref{2-11}, we obtain that
\begin{align}\label{2-12}
\left|z'(r)\right|
\leq&r^{-1}\int_0^r\left|z(\rho)+h(\rho)\right|\rho\mathrm{d}\rho
\nonumber
\\
\leq&r^{-1}\int_0^R\left(\left|z(\rho)\right|+\left|h(\rho)\right|\right)\rho\mathrm{d}\rho
\nonumber
\\
\leq& r^{-1}\omega_2^{-1}\left(\left\|z\right\|_{L^1(\Omega)}+\left\|h\right\|_{L^1(\Omega)}\right)
\nonumber
\\
=&2\omega_2^{-1}\left\|h\right\|_{L^1(\Omega)}r^{-1},
\quad\text{for }N=2\text{ and all }r\in(0,R],
\end{align}
where $\omega_2$ denotes the surface area of the unit sphere in $\mathbb{R}^2$. Setting $C_2=2\omega_2^{-1}$ completes the proof of \eqref{2-6}.
Employing the identity
\begin{equation*}
z(r)=-\int_r^R z'(\rho)\mathrm{d}\rho+z(R),
\quad\text{for all }r\in(0,R],
\end{equation*}
along with \eqref{2-12}, we deduce that
\begin{align*}
\left|z(r)\right|
\leq&\int_r^R\left|z'(\rho)\right|\mathrm{d}\rho+\left|z(R)\right|
\nonumber
\\
\leq&2\omega_2^{-1}\left\|h\right\|_{L^1(\Omega)}\left|\int_r^R\rho^{-1}\mathrm{d}\rho\right|
+\left|z(R)\right|
\nonumber
\\
=&2\omega_2^{-1}\left\|h\right\|_{L^1(\Omega)}\left|\ln R-\ln r\right|+\left|z(R)\right|
\nonumber
\\
\leq&C_1\left\|h\right\|_{L^1(\Omega)}\left(1+\left|\ln r\right|\right)+\left|z(R)\right|,
\quad\text{for all }r\in(0,R],
\end{align*}
with $C_1=2\omega_2^{-1}(1+\left|\ln R\right|)$.
\end{proof}

\section{Finite-time blow-up}

\subsection{A comparison lemma}

\begin{lemma}\label{le3.1}
Let $N\geq1$, $R>0$, and $M^*>M_*>0$. Suppose that $D\in C^3([0,\infty))$ fulfills \eqref{D}, and that the initial data $(u_0,w_0)$ satisfies \eqref{1-11}, together with
\begin{equation}\label{3-1}
M_*\leq\int_\Omega u_0\mathrm{d}x\leq M^*
\quad\text{and}\quad
M_*\leq\int_\Omega w_0\mathrm{d}x\leq M^*.
\end{equation}
Define
\begin{equation}\label{3-2}
\mu_*:=\frac{M_*}{2|\Omega|}
\quad\text{and}\quad
\mu^*:=CM^*,
\end{equation}
where $C>0$ is a constant depending only on $N$ and $R$.
For the solution $(u,v,w)$ in \eqref{wen}, we define the mass distribution functions
\begin{equation}\label{3-3}
U(s,t):=\int_0^{s^{\frac{1}{N}}}\rho^{N-1}u(\rho,t)\mathrm{d}\rho,
\quad s\in[0,R^N],~t\in[0,T_{\max}),
\end{equation}
and
\begin{equation}\label{3-4}
W(s,t):=\int_0^{s^{\frac{1}{N}}}\rho^{N-1}w(\rho,t)\mathrm{d}\rho,
\quad s\in [0,R^N],~t\in[0, T_{\max}),
\end{equation}
with the following properties:
\begin{itemize}
\item Regularity:
$$U\in C^0\left([0,T_{\max});C^1([0,R^N])\right)\cap C^1\left([0,R^N]\times(0,T_{\max})\right),$$
and
$$W,~W_t\in C^0\left([0,R^N]\times[0,T_{\max})\right).$$
\item Monotonicity and boundary values:
$$U_s\geq 0,$$
and for all $t\in(0,T_{\max})\cap(0,\ln 2)$,
\begin{equation}\label{3-5}
U(0,t)=W(0,t)=0,
\quad
U(R^N,t)\geq\frac{\mu_* R^N}{N},
\quad
W(R^N,t)\geq\frac{\mu_* R^N}{N}.
\end{equation}
\item Differential inequalities: For all $s\in(0,R^N)$ and $t\in(0,T_{\max})\cap(0,\ln 2)$,
\begin{equation}\label{3-6}
\mathcal{P}[U,W](s,t)\geq0,
\end{equation}
and
\begin{equation}\label{3-7}
\mathcal{Q}[U,W](s,t)=0.
\end{equation}
\end{itemize}
Hereafter, for any $T>0$, we consider functions $\varphi$ and $\psi$ defined on $[0,R^N]\times(0,T)$ with regularity specified as follows:
$$\varphi\in C^1\left([0,R^N]\times(0,T)\right)\quad\text{and}\quad\psi,~\psi_t\in C^0\left([0,R^N]\times(0,T)\right),$$
and
$$\varphi_s\geq 0,\quad\text{on~}(0,R^N)\times(0,T),$$
as well as
$$\varphi(\cdot,t)\in W_{loc}^{2,\infty}\left(\left(0,R^N\right)\right),
\quad\text{for all~}t\in(0,T).$$
For these functions, we let
\begin{align}\label{3-8}
&\mathcal{P}[\varphi,\psi](s,t)
\nonumber
\\
:=&\varphi_t
-N^2s^{2-\frac{2}{N}}D(N\varphi_s)\varphi_{ss}
-N\varphi_s\left(\psi-b(s)\right)\cdot(1+\left|\partial_rv(r,t)\right|^2)^\sigma,
\end{align}
and
\begin{equation}\label{3-9}
\mathcal{Q}[\varphi,\psi](s,t):=\psi_t+\psi-\varphi,
\end{equation}
for $t\in(0,T)$ and a.e. $s\in(0,R^N)$, where
$v$ is the solution to \eqref{wen}, and
\begin{equation}\label{3-10}
r:=\sqrt[N]{s},
\quad\text{for }s\in(0,R^N),
\end{equation}
as well as
\begin{align}\label{3-11}
b(s):=
\begin{cases}
\mu^* s,
&N=1,
\\
\mu^* \left(s+s^{\frac{3}{4}}+s^{\frac{3}{2}}\right),
&N=2,
\\
\mu^* s^{\frac{2}{N}},
&N\geq3.
\end{cases}
\end{align}
\end{lemma}

\begin{proof}
Applying the first inequality in \eqref{3-1} and \eqref{2-3} leads to
\begin{equation}\label{3-12}
\mu^{(u)}(t)=\frac{1}{\left|\Omega\right|}\int_\Omega u(\cdot,t)\mathrm{d}x
=\frac{1}{\left|\Omega\right|}\int_\Omega u_0\mathrm{d}x
\geq\frac{M_*}{|\Omega|}
\geq\frac{M_*}{2|\Omega|},
\end{equation}
for all $t\in(0,T_{\max})$. Integrate the third equation in \eqref{wen} to derive that
\begin{equation}\label{3-13}
\frac{\mathrm{d}}{\mathrm{d}t}
\mu^{(w)}(t)+\mu^{(w)}(t)
=\mu^{(u)}(t),
\quad\text{for~all}~t\in(0,T_{\max}),
\end{equation}
which gives rise to
\begin{equation*}
\mu^{(w)}(t)
=\left(1-e^{-t}\right)\mu^{(u)}(0)+e^{-t}\mu^{(w)}(0),
\quad\text{for~all~}t\in(0,T_{\max}).
\end{equation*}
Due to the nonnegativity of $\mu^{(u)}(0)$ and \eqref{3-1}, we assert that
\begin{equation}\label{3-14}
\mu^{(w)}(t)
\geq\frac{1}{2}\mu^{(w)}(0)
\geq\frac{M_*}{2|\Omega|},
\quad\text{for all}~t\in(0,T_{\max})\cap(0,\ln 2).
\end{equation}
From \eqref{wen}, we infer that the functions $U$ and $W$ defined in \eqref{3-3} and \eqref{3-4} indeed possess the asserted regularity and monotonicity, and further satisfy
\begin{equation}\label{3-15}
U(0,t)=W(0,t)=0,~
U(R^N,t)=\frac{\mu^{(u)}(t)R^N}{N},~
W(R^N,t)=\frac{\mu^{(w)}(t)R^N}{N},
\end{equation}
for all $t\in(0,T_{\max})$, and
\begin{equation}\label{3-16}
W_t=-W+U,
\quad\text{in }(0,R^N)\times(0,T_{\max}),
\end{equation}
as well as
\begin{equation}\label{3-17}
U_t=N^2s^{2-\frac{2}{N}}D(NU_s)U_{ss}
+NU_s\left(W-V\right)\cdot(1+\left|\partial_rv(r,t)\right|^2)^\sigma,
\end{equation}
in $(0,R^N)\times(0,T_{\max})$, with
\begin{equation}\label{3-18}
V(s,t):=\int_0^{s^{\frac{1}{N}}}\rho^{N-1}v(\rho,t)\mathrm{d}\rho,
\quad s\in [0,R^N],~t\in[0, T_{\max}).
\end{equation}
By means of
\begin{equation}\label{3-19}
w(x,t)
=w_0(x)e^{-t}
+\int_0^te^{-(t-s)}u(x,s)\mathrm{d}s,
\quad x\in\Omega,~t>0,
\end{equation}
and \citep[Remark 2.1]{NT}, we find that
\begin{align}\label{3-20}
&\left\|v(\cdot,t)\right\|_{L^\infty}
\nonumber
\\
\leq&C_5(R)\left\|w\right\|_{L^1}
\nonumber
\\
\leq&C_5(R)\left(\left\|u_0\right\|_{L^1}
+\left\|w_0\right\|_{L^1}\right),
\quad\text{for }N=1~\text{and for all}~t\in[0, T_{\max}),
\end{align}
where $C_5(R)>0$ is a constant depending only on $R$.
Recalling \eqref{3-1}, \eqref{3-18} and \eqref{3-20}, we gain that
\begin{equation}\label{3-21}
V(s,t)
\leq C_5(R)\left(\left\|u_0\right\|_{L^1(\Omega)}
+\left\|w_0\right\|_{L^1(\Omega)}\right)s
\leq 2C_5(R)M^* s,
\quad\text{for }N=1,
\end{equation}
and for all $s\in [0,R^N],~t\in[0, T_{\max})$.
For $s\in [0,R^N]$ and $t\in[0, T_{\max})$, relying on \eqref{2-5}, \eqref{3-1}, \eqref{3-18} and \eqref{3-19}, as well as the fact that
\begin{equation*}
\rho\left|\ln\rho\right|\leq
\begin{cases}
\rho^{\frac{1}{2}},
&\text{if}~\rho\in(0,1],
\\
\rho^2,
&\text{if}~\rho>1,
\end{cases}
\end{equation*}
we estimate that
\begin{align}\label{3-22}
&V(s,t)
\nonumber
\\
\leq&C_6(R)\left\|w\right\|_{L^1}
\int_0^{s^{\frac{1}{N}}}\rho\left(1+\left|\ln\rho\right|\right)\mathrm{d}\rho
+C_6(R)\int_0^{s^{\frac{1}{N}}}\rho\mathrm{d}\rho
\nonumber
\\
\leq&C_6(R)\left(\left\|u_0\right\|_{L^1(\Omega)}
+\left\|w_0\right\|_{L^1(\Omega)}\right)
\left(\int_0^{s^{\frac{1}{N}}}\rho\mathrm{d}\rho
+\int_0^{s^{\frac{1}{N}}}\rho^{\frac{1}{2}}\mathrm{d}\rho
+\int_0^{s^{\frac{1}{N}}}\rho^2\mathrm{d}\rho\right)
\nonumber
\\
&+C_6(R)s
\nonumber
\\
\leq&C_7(R)M^*
\left(s+s^{\frac{3}{4}}+s^{\frac{3}{2}}\right),
\quad\text{for }N=2,
\end{align}
and
\begin{align}\label{3-23}
V(s,t)
\leq&C_8(NR)\left\|w\right\|_{L^1(\Omega)}\int_0^{s^{\frac{1}{N}}}\rho\mathrm{d}\rho
\nonumber
\\
\leq&\frac{C_8(NR)}{2}\left(\left\|u_0\right\|_{L^1(\Omega)}+\left\|w_0\right\|_{L^1(\Omega)}\right)s^{\frac{2}{N}}
\nonumber
\\
\leq&C_8(NR)M^* s^{\frac{2}{N}},
\quad\text{for }N\geq3,
\end{align}
where $C_6(R),~C_7(R)>0$ and $C_8(NR)$ are constants depending only on $N$ and $R$. Let $C:=\max\{2C_5(R), C_7(R), C_8(NR)\}$.
Employing \eqref{3-17}, \eqref{3-21}-\eqref{3-23}, $U_s\geq 0$, as well as the definition of \eqref{3-2}, \eqref{3-8} and \eqref{3-11}, we conclude \eqref{3-6}. \eqref{3-7} is equivalent to \eqref{3-16} under the definition of \eqref{3-9}.
By means of the definition of $\mu_*$ in \eqref{3-2}, and \eqref{3-12}, \eqref{3-14}, \eqref{3-15}, we get \eqref{3-5}.
\end{proof}

\begin{lemma}\label{le3.2}
Let $N\geq1$, $R>0$ and $\mu^*>0$. Assume that $D\in C^3([0,\infty))$ satisfies \eqref{D} and that for some $T>0$, the functions $$\underline{U},~\overline{U}\in C^0\left([0,T);C^1([0,R^N])\right)\cap C^1\left([0,R^N]\times(0,T)\right),$$
and
$$\underline{W},~\overline{W}~\underline{W}_t,~\overline{W}_t~\in C^0\left([0,R^N]\times[0,T)\right),$$
fulfill
\begin{itemize}
\item For all $t\in(0,T)$,
$$\{\underline{U}(\cdot,t),~\overline{U}(\cdot,t)\}\subset W_{loc}^{2,\infty}\left((0,R^N)\right),$$
and
$$\underline{U}_s \geq0,~
\overline{U}_s \geq0,
\quad\text{on}~(0,R^N)\times(0,T).$$
\item The operators $\mathcal{P}$ and $\mathcal{Q}$ introduced in \eqref{3-8} and \eqref{3-9} satisfy
\begin{equation*}
\mathcal{P}[\underline{U},\underline{W}](s,t)
\leq0
\quad\text{and}\quad
\mathcal{P}[\overline{U},\overline{W}](s,t)
\geq0,
\end{equation*}
for all $t\in(0,T)$ and a.e. $s\in(0,R^N)$,
as well as
\begin{equation*}
\mathcal{Q}[\underline{U},\underline{W}](s,t)\leq0
\quad\text{and}\quad
\mathcal{Q}[\overline{U},\overline{W}](s,t)\geq0,
\end{equation*}
for all $t\in(0,T)$ and a.e. $s\in(0,R^N)$.
\item The boundary and initial conditions hold:
\begin{equation*}
\begin{cases}
\underline{U}(0,t)\leq\overline{U}(0,t),
\quad
\underline{U}(R^N,t)\leq\overline{U}(R^N,t),
\\
\underline{W}(0,t)\leq\overline{W}(0,t),
\quad
\underline{W}(R^N,t)\leq\overline{W}(R^N,t),
\end{cases}
\end{equation*}
for all $t\in(0,T)$, and
\begin{equation*}
\underline{U}(s,0)\leq\overline{U}(s,0)
\quad\text{and}\quad
\underline{W}(s,0)\leq\overline{W}(s,0),
\quad\text{for all}~s\in[0,R^N].
\end{equation*}
\end{itemize}
Then
\begin{equation*}
\underline{U}(s,t)\leq\overline{U}(s,t),
\quad\text{for all}~s\in[0,R^N]~\text{and}~t\in[0,T).
\end{equation*}
\end{lemma}

\begin{proof}
It is worth noting that in the derivation of $\mathcal{P}(s,t)$ in \eqref{3-8}, the variable $V$ in equation \eqref{3-17} is estimated by $b(s)$, which is a function depending solely on the spatial variable. Moreover, in contrast to \cite{JLY,ZMT}, we retain $\nabla v$ rather than expressing it in terms of $W$. This treatment exempts the term $\nabla v$ from being involved in the comparison of sub- and supersolutions. Consequently, this approach enables us to break through the parameter restriction $\sigma \leq 0$ encountered in \cite{JLY,ZMT}. By the definition of
\begin{equation*}
\varphi(s,t):=\underline{U}(s,t)-\overline{U}(s,t)-\varepsilon e^{\vartheta t},\quad(s,t)\in[0,R^N]\times[0,T'),
\end{equation*}
in \cite[(3.25)]{TW} and $\mathcal{P}[\underline{U},\underline{W}](s,t)
\leq0\leq
\mathcal{P}[\overline{U},\overline{W}](s,t)$, we know that
\begin{align*}
\varphi_t
=&\underline{U}_t-\overline{U}_t-\vartheta\varepsilon e^{\vartheta t}
\\
\leq&N^2s^{2-\frac{2}{N}}D(N\underline{U}_s)\underline{U}_{ss}
-N^2s^{2-\frac{2}{N}}D(N\overline{U}_s)\overline{U}_{ss}
-\vartheta\varepsilon e^{\vartheta t}
\\
&
+N\underline{U}_s\left(\underline{W}-b(s)\right)\left(1+\left|\partial_r v(r,t)\right|^2\right)^\sigma
-N\overline{U}_s\left(\overline{W}-b(s)\right)\left(1+\left|\partial_r v(r,t)\right|^2\right)^\sigma
\\
=&N^2s^{2-\frac{2}{N}}D(N\underline{U}_s)\varphi_{ss}
+N^2s^{2-\frac{2}{N}}\left\lbrace D(N\underline{U}_s)-D(N\overline{U}_s)\right\rbrace\overline{U}_{ss}
-\vartheta\varepsilon e^{\vartheta t}
\\
&
+N\underline{U}_s\left(\underline{W}-b(s)\right)\left(1+\left|\partial_r v(r,t)\right|^2\right)^\sigma
-N\overline{U}_s\left(\overline{W}-b(s)\right)\left(1+\left|\partial_r v(r,t)\right|^2\right)^\sigma.
\end{align*}
Due to $\varphi_t(s_0,t_0)\geq 0$, $\varphi_s(s_0,t_0)=0$ and $\liminf\limits_{j\to\infty} \varphi_{ss}(s_j, t_0) \leq 0$ \cite[(3.27), (3.29) and (3.30)]{TW}, we conclude that
\begin{align*}
0\leq N\underline{U}_s(s_0,t_0)\left(\underline{W}(s_0,t_0)-\overline{W}(s_0,t_0)\right)\left(1+\left|\partial_r v(r_0,t_0)\right|^2\right)^\sigma
-\vartheta\varepsilon e^{\vartheta t_0}.
\end{align*}
Since the arguments are essentially identical to those in \cite[Lemma 3.2]{TW} except for the aforementioned modifications, we omit the remaining details. Note that the operators $\mathcal{P}$ and $\mathcal{Q}$ involve no derivatives of $\underline{W},~\overline{W}$ with respect to $s$, then we only need $\underline{W},~\overline{W}~\underline{W}_t,~\overline{W}_t~\in C^0\left([0,R^N]\times[0,T)\right)$ throughout the proof.
\end{proof}

\subsection{Construction of blow-up subsolutions.}

\begin{lemma}{\rm(\citep[Lemma 3.3]{TW})}\label{le3.3}
Let $N\geq1$, $R>0$, $\mu_*,~\mu^*>0$, $\varepsilon\in(0,\frac{1}{4})$, and
\begin{equation}\label{d}
d
:=\frac{\mu_* R^N}{N\left(e^{\frac{1}{e}}(R^N+1)+\mu^*+R^{N\varepsilon}\right)}.
\end{equation}
For all $T>0$ and any function $g\in C^1([0,T))$ satisfying $g(t)>\frac{1}{R^N}$, we define
\begin{equation}\label{3-25}
\widehat{U}(s,t):=
\begin{cases}
dg^{1-\varepsilon}(t)s,
&t\in[0,T),~s\in[0,\frac{1}{g(t)}],
\\
\varepsilon^{-\varepsilon}d\left(s-\frac{1-\varepsilon}{g(t)}\right)^\varepsilon,
&t\in[0,T),~s\in(\frac{1}{g(t)},R^N].
\end{cases}
\end{equation}
Then we derive that
\begin{itemize}
\item Regularity and boundary values:
$$\widehat{U}\in C^1\left([0,R^N]\times[0,T)\right)\cap C^0\left([0,T);W^{2,\infty}((0,R^N))\right),$$
and
\begin{equation*}
\widehat{U}(0,t)=0
\quad\text{and}\quad
\widehat{U}(R^N,t)\leq\frac{\mu_* R^N}{N},
\quad\text{for all}~t\in(0,T).
\end{equation*}
\item For all $t\in(0,T)$, the function $\widehat{U}(\cdot,t)\in C^2\left([0,R^N]\right)\setminus\{\frac{1}{g(t)}\}$
satisfies
\begin{equation}\label{3-26}
\widehat{U}_t(s,t)=
\begin{cases}
(1-\varepsilon)dg^{-\varepsilon}(t)g'(t)s,
&s\in(0,\frac{1}{g(t)}),
\\
\varepsilon^{1-\varepsilon}(1-\varepsilon)d\left(s-\frac{1-\varepsilon}{g(t)}\right)^{\varepsilon-1}\cdot\dfrac{g'(t)}{g^2(t)},
&s\in(\frac{1}{g(t)},R^N),
\end{cases}
\end{equation}
and
\begin{equation}\label{3-27}
\widehat{U}_{s}(s,t)=
\begin{cases}
dg^{1-\varepsilon}(t),
&s\in(0,\frac{1}{g(t)}),
\\
\varepsilon^{1-\varepsilon}d\left(s-\frac{1-\varepsilon}{g(t)}\right)^{\varepsilon-1},
&s\in(\frac{1}{g(t)},R^N),
\end{cases}
\end{equation}
as well as
\begin{equation}\label{3-28}
\widehat{U}_{ss}(s,t)=
\begin{cases}
0,
&s\in(0,\frac{1}{g(t)}),
\\
-\varepsilon^{1-\varepsilon}(1-\varepsilon)d\left(s-\frac{1-\varepsilon}{g(t)}\right)^{\varepsilon-2},
&s\in(\frac{1}{g(t)},R^N).
\end{cases}
\end{equation}
\end{itemize}
\end{lemma}

\begin{lemma}\label{le3.4}
Let $N\geq1$, $R>0$ and $\mu_*,~\mu^*>0$ hold.
For each $\varepsilon\in(0,\frac{1}{4})$, we propose the function $\widehat{W}(s)$ as
\begin{equation}\label{3-29}
\widehat{W}(s)=ds^\varepsilon,
\quad\text{for all }s\in[0,R^N],
\end{equation}
where $d$ is shown in \eqref{d}.
Then we yield that
$$\widehat{W},~\widehat{W}_t\in C^0\left([0,R^N]\times[0,T)\right),$$
and
\begin{equation*}
\widehat{W}(0,t)=0
\quad\text{and}\quad
\widehat{W}(R^N,t)\leq\frac{\mu_* R^N}{N},
\quad\text{for all}~t\in(0,T),
\end{equation*}
as well as
\begin{equation}\label{3-30}
\widehat{W}_{t}(s,t)\equiv0,
\quad\text{for all }t\in[0,T),~s\in[0,R^N].
\end{equation}
\end{lemma}

\begin{lemma}\label{le3.5}
Let $N\geq1$, $R>0$, $\mu_*,~\mu^*>0$ and $\varepsilon\in(0,\frac{1}{4})\cap(0,\frac{1}{N})$. Then we conclude that
\begin{equation}\label{3-31}
\widehat{W}(s)-eb(s)\geq\frac{d}{2}s^\varepsilon,
\quad\text{for all }s\in[0,s_*],
\end{equation}
where $b(s)$, $d$ and $\widehat{W}$ are as in \eqref{3-11}, \eqref{d} and Lemma \ref{le3.4} separately, and
\begin{align}\label{3-32}
s_*:=\min\Biggl\{1,
\left(\frac{d}{2e\mu^*}\right)^{\frac{1}{1-\varepsilon}},
\left(\frac{d}{6e\mu^*}\right)^{\frac{1}{\frac{3}{4}-\varepsilon}},
\left(\frac{d}{2e\mu^*}\right)^{\frac{1}{\frac{2}{N}-\varepsilon}}\Biggr\}.
\end{align}
\end{lemma}

\begin{proof}
For $N=2$, using \eqref{3-11}, \eqref{3-29} and \eqref{3-32}, we derive that
\begin{align*}
\widehat{W}(s,t)-eb(s)
=&ds^\varepsilon
-e\mu^*\left(s+s^{\frac{3}{4}}+s^{\frac{3}{2}}\right)
\\
\geq&\frac{d}{2}s^\varepsilon
+s^\varepsilon(\frac{d}{2}-3e\mu^*s^{\frac{3}{4}-\varepsilon})
\\
\geq&\frac{d}{2}s^\varepsilon
+s^\varepsilon(\frac{d}{2}-3e\mu^*s_*^{\frac{3}{4}-\varepsilon})
\\
\geq&0,
\quad\text{for all}~s\in[0,s_*].
\end{align*}
Similarly, we can obtain the conclusion for $N=1$ and $N\geq3$.
\end{proof}

For any fixed $\theta>0$ and all $s\in[0,R^N]$, $t\in[0,T)$, we introduce
\begin{equation}\label{3-33}
\underline{U}(s,t)
:=e^{-\theta t}\widehat{U}(s,t)
\quad\text{and}\quad
\underline{W}(s,t)
:=e^{-\theta t}\widehat{W}(s),
\end{equation}
which will be shown to be a subsolution to \eqref{wen}.

\subsection{Subsolution properties: $\mathcal{P}[\underline{U},\underline{W}](s,t)\leq 0$.}

\begin{lemma}\label{le3.6}
Let $N\geq1$, $R>0$, $T>0$, $\varepsilon\in(0,\frac{1}{4})\cap(0,\frac{1}{N})$,  and let $D\in C^3([0,\infty))$ fulfill \eqref{D}. Then for any $\theta>0$, the functions $\underline{U}$, $\underline{W}$ and $\widehat{U}$, $\widehat{W}$ as in \eqref{3-33} and  Lemma \ref{le3.3}, Lemma \ref{le3.4} satisfy
\begin{equation}\label{3-34}
\mathcal{P}[\underline{U},\underline{W}](s,t)
\leq e^{-\theta t}\cdot
\begin{cases}
\mathcal{P}_1[\widehat{U},\widehat{W}](s,t),
&\sigma\geq 0,
\\
\mathcal{P}_2[\widehat{U},\widehat{W}](s,t),
&\sigma<0,
\end{cases}
\end{equation}
for all $t\in(0,T)\cap(0,\frac{1}{\theta})$ and $s\in[0,s_*]$,
where
\begin{align}\label{P1}
\mathcal{P}_1[\widehat{U},\widehat{W}](s,t)
=\widehat{U}_t
-\theta\widehat{U}-N^2s^{2-\frac{2}{N}}D(Ne^{-\theta t}\widehat{U}_s)\widehat{U}_{ss}
-\frac{Nd}{2e}\widehat{U}_s s^\varepsilon,
\end{align}
and
\begin{align}\label{P2}
\mathcal{P}_2[\widehat{U},\widehat{W}](s,t)
=&\widehat{U}_t-\theta\widehat{U}-N^2s^{2-\frac{2}{N}}D(Ne^{-\theta t}\widehat{U}_s)\widehat{U}_{ss}
\nonumber
\\
&-\frac{Nd}{2e}\widehat{U}_s s^\varepsilon\cdot
\begin{cases}
C_9,
&N=1,
\\
C_{10} s^{2\sigma(\frac{1}{N}-1)},
&N\geq2.
\end{cases}
\end{align}
Here, $C_9,~C_{10}>0$ are constants, together with $d$ and $s_*$ are defined in \eqref{d} and \eqref{3-32}, respectively.
\end{lemma}

\begin{proof}
We restrict our proof to the case $\sigma<0$, and the conclusion for $\sigma\geq 0$ follows analogously.
For $N=1$, according to \citep[Remark 2.1]{NT}, we deduce that
\begin{equation}\label{3-37}
\left|\partial_rv(r,t)\right|\leq c_1,
\quad\text{for all } t\in(0,T)\text{ and }s\in[0,R^N].
\end{equation}
By means of \eqref{3-8}, \eqref{3-10}, \eqref{3-11}, \eqref{3-31}, \eqref{3-33}, \eqref{P2}, \eqref{3-37}, $\sigma<0$, $\theta>0$, Lemma \ref{le2.3}, and the nonnegativity of $\widehat{U}_s$ and $\widehat{W}$, we achieve that
\begin{align*}
&\mathcal{P}[\underline{U},\underline{W}](s,t)
\nonumber
\\
=&\underline{U}_{t}
-N^2s^{2-\frac{2}{N}}D(N\underline{U}_{s})\underline{U}_{ss}
-N\underline{U}_{s}\left(\underline{W}-b(s)\right)\cdot(1+\left|\partial_rv(r,t)\right|^2)^\sigma
\nonumber
\\
=&e^{-\theta t}\cdot\Biggl\{\widehat{U}_t-\theta\widehat{U}-N^2s^{2-\frac{2}{N}}D(Ne^{-\theta t}\widehat{U}_s)\widehat{U}_{ss}
\nonumber
\\
&-N\widehat{U}_s\left(e^{-\theta t}\widehat{W}-b(s)\right)\cdot(1+\left|\partial_rv(r,t)\right|^2)^\sigma\Biggr\}
\nonumber
\\
\leq&e^{-\theta t}\cdot\Biggl\{\widehat{U}_t-\theta\widehat{U}-N^2s^{2-\frac{2}{N}}D(Ne^{-\theta t}\widehat{U}_s)\widehat{U}_{ss}\Biggr\}
\nonumber
\\
&-e^{-\theta t}N\widehat{U}_se^{-1}\left(\widehat{W}-eb(s)\right)\cdot
\begin{cases}
c_2^\sigma,
&N=1,
\\
\left(1+c_3r^{-2}\right)^\sigma,
&N=2,
\\
\left(1+c_4r^{2-2N}\right)^\sigma,
&N\geq3,
\end{cases}
\nonumber
\\
\leq&e^{-\theta t}\cdot\Biggl\{\widehat{U}_t-\theta\widehat{U}-N^2s^{2-\frac{2}{N}}D(Ne^{-\theta t}\widehat{U}_s)\widehat{U}_{ss}\Biggr\}
\nonumber
\\
&-e^{-\theta t}\frac{Nd}{2e}\widehat{U}_s s^\varepsilon\cdot
\begin{cases}
c_2^\sigma,
&N=1,
\\
c_5^\sigma s^{2\sigma(\frac{1}{N}-1)},
&N\geq2,
\end{cases}
\nonumber
\\
=&e^{-\theta t}\cdot\mathcal{P}_2[\widehat{U},\widehat{W}](s,t),
\quad\text{for all}~t\in(0,T)\cap(0,\frac{1}{\theta}) \text{ and }s\in[0,s_*].
\end{align*}
\end{proof}

\subsubsection{$\mathcal{P}[\underline{U},\underline{W}](s,t)\leq 0$ for $\sigma\geq 0$ in inner and intermediate interval.}\label{sub3.3.1}

\begin{lemma}\label{le3.7}
Let $N\geq1$, $R>0$, $\varepsilon\in(0,\frac{1}{4})\cap(0,\frac{1}{N})$ and $\sigma\geq 0$. Suppose that $D\in C^3([0,\infty))$ satisfies \eqref{D} and that $s_*$ is shown in \eqref{3-32}.
Then for any fixed $\mu_*,~\mu^*>0$, one can select
\begin{equation}\label{3-38}
g_*>\max\Biggl\{\frac{1}{R^N},~\frac{1}{s_*}\Biggr\},
\end{equation}
and $\gamma_*>0$ satisfying the properties below:\\
If $T>0$ and $g\in C^1([0,T))$ satisfy
\begin{equation}\label{3-39}
g(t)\geq g_*,
\quad\text{for all~}t\in(0,T),
\end{equation}
and
\begin{equation}\label{3-40}
0\leq g'(t)\leq\gamma_* g^{1+(1-\varepsilon)}(t),
\quad\text{for all~}t\in(0,T),
\end{equation}
then for any
\begin{equation}\label{3-41}
\theta>0,
\end{equation}
the functions $\underline{U}$ and $\underline{W}$ defined in \eqref{3-33} satisfy
\begin{equation}\label{3-42}
\mathcal{P}[\underline{U},\underline{W}](s,t)
\leq0,
\quad\text{for all}~t\in(0,T)\cap(0,\frac{1}{\theta})
~\text{and}~s\in(0,\frac{1}{g(t)}).
\end{equation}
\end{lemma}

\begin{proof}
Making use of Lemma \ref{le3.3}, \eqref{P1}, \eqref{3-41} and $\varepsilon<1$, we estimate that
\begin{align}\label{3-43}
\mathcal{P}_1[\widehat{U},\widehat{W}](s,t)
=&\widehat{U}_t
-\theta\widehat{U}-N^2s^{2-\frac{2}{N}}D(Ne^{-\theta t}\widehat{U}_s)\widehat{U}_{ss}
-\frac{Nd}{2e}\widehat{U}_s s^\varepsilon
\nonumber
\\
\leq&\widehat{U}_t
-\frac{Nd}{2e}\widehat{U}_s s^\varepsilon
\nonumber
\\
\leq&dg^{-\varepsilon}(t)s\cdot \left(g'(t)
-\frac{Nd}{2e}g(t)s^{\varepsilon-1}\right)
\nonumber
\\
\leq&dg^{-\varepsilon}(t)s\cdot \left(g'(t)
-\frac{Nd}{2e}g^{1+(1-\varepsilon)}(t)\right),
\end{align}
for all $t\in(0,T)$ and $s\in\left(0,\frac{1}{g(t)}\right)$.
Let $\gamma_*:=\frac{Nd}{2e}$ and
use \eqref{3-34}, \eqref{3-38}, \eqref{3-40}, \eqref{3-43} to arrive at \eqref{3-42}.
\end{proof}

\begin{lemma}\label{le3.8}
Let $N\geq1$, $R>0$ and $\sigma\geq 0$. Suppose that $D\in C^3([0,\infty))$ satisfies \eqref{D}, \eqref{1-8} with parameters $\xi_0>0$, $K_D>0$, and exponent $m\in\mathbb{R}$ fulfilling
\begin{equation}\label{3-44}
m<2-\frac{2}{N}.
\end{equation}
Then there exists a $\varepsilon_*\in(0,\frac{1}{4})\cap(0,\frac{1}{N})$ such that for every $\varepsilon\in(0,\varepsilon_*)$ and $\mu_*,~\mu^*>0$, we can choose  $s_{**}\in(0,1]$ and $\gamma_{**}>0$ with the properties as follows:\\
Assume that $T>0$ and $g\in C^1([0,T))$ fulfill
\begin{equation}\label{3-45}
g(t)\geq g_*,
\quad\text{for all}~t\in(0,T),
\end{equation}
and
\begin{equation}\label{3-46}
0\leq g'(t)
\leq\gamma_{**}g^{1+(1-\varepsilon)}(t),
\quad\text{for all}~t\in(0,T),
\end{equation}
where $g_*$ is given in Lemma \ref{le3.7}.
Then for any
\begin{equation}\label{3-47}
\theta>0,
\end{equation}
the functions $\underline{U}$ and $\underline{W}$ defined in \eqref{3-33} satisfy
\begin{equation}\label{3-48}
\mathcal{P}[\underline{U},\underline{W}](s,t)
\leq0,
\end{equation}
for all $t\in(0,T)\cap(0,\frac{1}{\theta})$ and $s\in(\frac{1}{g(t)},R^N)\cap(0,s_{**})$.
\end{lemma}

\begin{proof}
Applying \eqref{3-44}, we can choose a sufficiently small $\varepsilon_*\in(0,\frac{1}{4})\cap(0,\frac{1}{N})$ such that
\begin{equation}\label{3-49}
-\frac{2}{N}+(\varepsilon-1)(m-2)>0,
\quad\text{for all }\varepsilon\in(0,\varepsilon_*).
\end{equation}
For each chosen $\varepsilon\in(0,\varepsilon_*)$, $\mu_*,~\mu^*>0$, as well as $d$ and $s_*$ taken from \eqref{d} and \eqref{3-32} separately, we use \eqref{3-49} and $\varepsilon<1$ to pick $s_{**}\in(0,1]$ and $\gamma_{**}>0$ properly small such that
\begin{equation}\label{3-50}
s_{**}\leq s_*,
\end{equation}
and
\begin{equation}\label{3-51}
\frac{N\varepsilon^{1-\varepsilon}ds_{**}^{\varepsilon-1}}{e}
>\xi_0,
\end{equation}
and
\begin{align}\label{3-52}
2N^mK_D\left(e+e^{2-m}\right)d^{m-2}(1-\varepsilon)\varepsilon^{(1-\varepsilon)(m-1)}
c_6(\varepsilon)\cdot s_{**}^{-\frac{2}{N}+(\varepsilon-1)(m-2)}
\leq\frac{1}{2},
\end{align}
as well as
\begin{equation}\label{3-53}
\frac{2e(1-\varepsilon)}{Nd}\gamma_{**}
\leq\frac{1}{2},
\end{equation}
with
\begin{equation}\label{3-54}
c_6(\varepsilon):=\max
\begin{Bmatrix}
1,\varepsilon^{(\varepsilon-1)(m-2)+\varepsilon-2}
\end{Bmatrix}.
\end{equation}
Note that
\begin{equation}\label{3-55}
s\geq
s-\frac{1-\varepsilon}{g(t)}
\geq s-(1-\varepsilon)s
=\varepsilon s,
\quad\text{for~all}~t\in(0,T)\text{~and~}
s\in(\frac{1}{g(t)},R^N).
\end{equation}
On the basis of \eqref{3-27}, \eqref{3-51}, \eqref{3-55}, as well as $\varepsilon<1$, we conclude that
\begin{equation}\label{3-56}
Ne^{-\theta t}\widehat{U}_s
\geq\frac{N}{e}\widehat{U}_s
=\frac{N}{e}\cdot\varepsilon^{1-\varepsilon}d\left(s-\frac{1-\varepsilon}{g(t)}\right)^{\varepsilon-1}
\geq\frac{N}{e}\cdot\varepsilon^{1-\varepsilon}ds_{**}^{\varepsilon-1}
>\xi_0,
\end{equation}
for all $t\in(0,T)\cap(0,\frac{1}{\theta})$ and $s\in(\frac{1}{g(t)},R^N)\cap(0,s_{**})$.
Due to \eqref{1-8}, \eqref{P1}, \eqref{3-47}, \eqref{3-56}, the fact that $\widehat{U}_{ss}\leq 0$ by \eqref{3-28}, and the nonnegativity of $\widehat{U}$, we assert that
\begin{align}\label{3-57}
\mathcal{P}_1[\widehat{U},\widehat{W}](s,t)
=&\widehat{U}_t
-\theta\widehat{U}-N^2s^{2-\frac{2}{N}}D(Ne^{-\theta t}\widehat{U}_s)\widehat{U}_{ss}
-\frac{Nd}{2e}\widehat{U}_s s^\varepsilon
\nonumber
\\
\leq&\widehat{U}_t
-N^2s^{2-\frac{2}{N}}K_D\cdot\left(Ne^{-\theta t}\widehat{U}_s\right)^{m-1}\widehat{U}_{ss}
-\frac{Nd}{2e}\widehat{U}_s s^\varepsilon,
\end{align}
for all $t\in(0,T)$ and $s\in(\frac{1}{g(t)},R^N)$.
Thanks to \eqref{3-27} and \eqref{3-28}, we find that
\begin{align}\label{3-58}
&\frac{-N^2s^{2-\frac{2}{N}}K_D\cdot(Ne^{-\theta t}\widehat{U}_s)^{m-1}\widehat{U}_{ss}}
{\frac{Nd}{2e}\widehat{U}_s\cdot s^{\varepsilon}}
\nonumber
\\
=&-\frac{2N^mK_De^{1-(m-1)\theta t}}{d}\cdot s^{2-\frac{2}{N}-\varepsilon}\widehat{U}_s^{m-2}\widehat{U}_{ss}
\nonumber
\\
=&2N^mK_De^{1-(m-1)\theta t}d^{m-2}(1-\varepsilon)\varepsilon^{(1-\varepsilon)(m-1)}
\nonumber
\\
&\times s^{2-\frac{2}{N}-\varepsilon}\left(s-\frac{1-\varepsilon}{g(t)}\right)^{(\varepsilon-1)(m-2)+\varepsilon-2},
\end{align}
for all $t\in(0,T)$ and $s\in(\frac{1}{g(t)},R^N)$.
Invoking \eqref{3-49}, \eqref{3-54} and \eqref{3-55}, we see that
\begin{align}\label{3-59}
s^{2-\frac{2}{N}-\varepsilon}\left(s-\frac{1-\varepsilon}{g(t)}\right)^{(\varepsilon-1)(m-2)+\varepsilon-2}
\leq&c_6(\varepsilon)\cdot s^{-\frac{2}{N}+(\varepsilon-1)(m-2)}
\nonumber
\\
\leq&c_6(\varepsilon)\cdot s_{**}^{-\frac{2}{N}+(\varepsilon-1)(m-2)},
\end{align}
for all $\varepsilon\in(0,\varepsilon_*)$, $t\in(0,T)$ and $s\in(\frac{1}{g(t)},R^N)\cap(0,s_{**})$.
Note that
\begin{equation}\label{3-60}
e^{1-(m-1)\theta t}\leq
\begin{cases}
e^{2-m},
&\text{if}~m<1,
\\
e,
&\text{if}~m\geq1,
\end{cases}
\end{equation}
for all $t\in(0,T)\cap(0,\frac{1}{\theta})$.
According to \eqref{3-52} and \eqref{3-58}-\eqref{3-60}, we estimate that
\begin{equation}\label{3-61}
-N^2s^{2-\frac{2}{N}}K_D\cdot(Ne^{-\theta t}\widehat{U}_s)^{m-1}\widehat{U}_{ss}
\leq\frac{1}{2}\cdot\frac{Nd}{2e}\widehat{U}_s\cdot s^{\varepsilon},
\end{equation}
for all $t\in(0,T)\cap(0,\frac{1}{\theta})$ and $s\in(\frac{1}{g(t)},R^N)\cap(0,s_{**})$.
Recalling \eqref{3-26} and \eqref{3-27}, we derive that
\begin{align}\label{3-62}
\frac{\widehat{U}_t}
{\frac{Nd}{2e}\widehat{U}_s\cdot s^{\varepsilon}}
=\frac{2e(1-\varepsilon)}{Nd}
\cdot
s^{-\varepsilon}
\cdot
\frac{g'(t)}{g^2(t)}
\leq\frac{2e(1-\varepsilon)}{Nd}
\cdot g^{\varepsilon-2}(t)g'(t),
\end{align}
for all $t\in(0,T)$ and $s\in(\frac{1}{g(t)},R^N)$.
Due to \eqref{3-46} and \eqref{3-53}, we establish that
\begin{equation}\label{3-63}
\frac{2e(1-\varepsilon)}{Nd}
\cdot g^{\varepsilon-2}(t)g'(t)
\leq\frac{2e(1-\varepsilon)}{Nd}\gamma_{**}
\leq\frac{1}{2},
\end{equation}
for all $t\in(0,T)$.
\eqref{3-62} and \eqref{3-63} imply
\begin{equation}\label{3-64}
\widehat{U}_t
\leq\frac{1}{2}\cdot\frac{Nd}{2e}\widehat{U}_s\cdot s^{\varepsilon},
\quad\text{for~all }t\in(0,T)\text{ and }s\in(\frac{1}{y(t)},R^N).
\end{equation}
In combination with \eqref{3-34}, \eqref{3-57}, \eqref{3-61} and \eqref{3-64} shows \eqref{3-48}.
\end{proof}

\subsubsection{$\mathcal{P}[\underline{U},\underline{W}](s,t)\leq 0$ for $\sigma<0$ in inner and intermediate interval.}\label{sub3.3.2}

\begin{remark}\label{remark 3.1}
{\rm $\mathcal{P}_1$ and $\mathcal{P}_2$ are analogous for $N=1$. Then $\mathcal{P}_2\leq 0$  for $N=1$ follows directly from the method in Section 3.3.1. We thus consider only the case $N\geq 2$ hereinafter.}
\end{remark}

\begin{lemma}\label{le3.9}
Let $N\geq2$, $R>0$, and
\begin{equation}\label{3-65}
\frac{N}{2-2N}<\sigma<0.
\end{equation}
Assume that $D\in C^3([0,\infty))$ satifies \eqref{D}.
Then for any fixed $\mu_*,~\mu^*>0$, there exist $\varepsilon_{**}\in(0,\frac{1}{4})\cap(0,\frac{1}{N})$, $\delta_1>0$ and $\tilde{\gamma}_*>0$, such that
if $T>0$ and $g\in C^1([0,T))$ satisfy
\begin{equation}\label{3-66}
g(t)\geq g_*,
\quad\text{for all~}t\in(0,T),
\end{equation}
and
\begin{equation}\label{3-67}
0\leq g'(t)\leq\tilde{\gamma}_* g^{1+\delta_1}(t),
\quad\text{for all~}t\in(0,T),
\end{equation}
where $g_*$ is defined in Lemma \ref{le3.7},
then for all $\varepsilon\in(0,\varepsilon_{**})$, and
\begin{equation}\label{3-68}
\theta>0,
\end{equation}
the functions $\underline{U}$ and $\underline{W}$ defined in \eqref{3-33} satisfy
\begin{equation}\label{3-69}
\mathcal{P}[\underline{U},\underline{W}](s,t)\leq0,
\quad\text{for all}~t\in(0,T)\cap(0,\frac{1}{\theta})
~\text{and}~s\in(0,\frac{1}{g(t)}).
\end{equation}
\end{lemma}

\begin{proof}
According to \eqref{3-65}, there exists a constant $\varepsilon_{**}\in(0,\frac{1}{4})\cap(0,\frac{1}{N})$ such that
\begin{equation}\label{3-70}
1-\varepsilon-2\sigma\left(\frac{1}{N}-1\right):=\delta_1
>0,
\quad\text{for all }\varepsilon\in(0,\varepsilon_{**}).
\end{equation}
We use Lemma \ref{le3.3}, \eqref{P2}, \eqref{3-67}, \eqref{3-68}, \eqref{3-70}, and the nonnegativity of $\widehat{U}$, as well as define $\tilde{\gamma}_*:=\frac{NdC_{10}}{2e}$ to achieve that
\begin{align}\label{3-71}
&\mathcal{P}_2[\widehat{U},\widehat{W}](s,t)
\nonumber
\\
=&\widehat{U}_t-\theta\widehat{U}-N^2s^{2-\frac{2}{N}}D(Ne^{-\theta t}\widehat{U}_s)\widehat{U}_{ss}
-\frac{Nd}{2e}C_{10}\widehat{U}_s\cdot
s^{\varepsilon+2\sigma(\frac{1}{N}-1)}
\nonumber
\\
\leq&\widetilde{U}_t
-\frac{Nd}{2e}C_{10}\widehat{U}_s\cdot
s^{\varepsilon+2\sigma(\frac{1}{N}-1)}
\nonumber
\\
\leq&dg^{-\varepsilon}s\left(g'(t)
-\frac{Nd}{2e}C_{10} g(t)s^{\varepsilon+2\sigma(\frac{1}{N}-1)-1}\right)
\nonumber
\\
\leq&dg^{-\varepsilon}(t)s\left(g'(t)-\frac{Nd}{2e}C_{10}g^{2-2\sigma(\frac{1}{N}-1)-\varepsilon}(t)\right)
\nonumber
\\
\leq&0,
\quad\text{for all } t\in(0,T),~s\in(0,\frac{1}{g(t)})\text{ and any }\varepsilon\in(0,\varepsilon_{**}).
\end{align}
\eqref{3-34} and \eqref{3-71} shows \eqref{3-69}.
\end{proof}

\begin{lemma}\label{le3.10}
Let $N\geq2$ and $R>0$. Assume that $D\in C^3([0,\infty))$ fulfills \eqref{D}, \eqref{1-8} with parameters $\xi_0>0$, $K_D>0$, and exponents $m\in\mathbb{R}$, $\sigma\in\mathbb{R}$ satisfying
\begin{equation}\label{3-72}
m<2-\frac{2}{N},
\end{equation}
and
\begin{equation}\label{3-73}
0>\sigma>\max\{\frac{N}{2-2N},\frac{mN+2-2N}{2N-2}\}.
\end{equation}
Then for any fixed $\mu_*,~\mu^*>0$, there exist $\varepsilon_{***}\in(0,\frac{1}{4})\cap(0,\frac{1}{N})$, $\delta_2>0$, $\tilde{s}_{**}\in(0,1]$ and $\tilde{\gamma}_{**}>0$, such that
if $T>0$ and $g\in C^1([0,T))$ satisfy
\begin{equation}\label{3-74}
g(t)\geq g_*,
\quad\text{for all}~t\in(0,T),
\end{equation}
and
\begin{equation}\label{3-75}
0\leq g'(t)
\leq\tilde{\gamma}_{**}g^{1+\delta_2}(t),
\quad\text{for all}~t\in(0,T),
\end{equation}
where $g_*$ is given in Lemma \ref{le3.7},
then for all $\varepsilon\in(0,\varepsilon_{***})$, and
\begin{equation}\label{3-76}
\theta>0,
\end{equation}
the functions $\underline{U}$ and $\underline{W}$ defined in \eqref{3-33} satisfy
\begin{equation}\label{3-77}
\mathcal{P}[\underline{U},\underline{W}](s,t)\leq0,
\quad\text{for all}~t\in(0,T)\cap(0,\frac{1}{\theta})
~\text{and}~s\in(\frac{1}{g(t)},R^N)\cap(0,\tilde{s}_{**}).
\end{equation}
\end{lemma}

\begin{proof}
By using \eqref{3-73}, we deduce that there exists $\varepsilon_{***}\in(0,1)$
such that
\begin{equation}\label{3-78}
-\frac{2}{N}+2\sigma(1-\frac{1}{N})+(\varepsilon-1)(m-2)>0,
\end{equation}
and
\begin{equation}\label{3-79}
1+2\sigma(1-\frac{1}{N})-\varepsilon:=\delta_2>0,
\end{equation}
for all $\varepsilon\in(0,\varepsilon_{***})$.
For every $\varepsilon\in(0,\varepsilon_{***})$, $\mu_*,~\mu^*>0$, $d$ and $s_{**}$ taken from \eqref{d} and \eqref{3-50}, we employ \eqref{3-78} and $\varepsilon<1$ to select $\tilde{s}_{**}\in(0,1]$ and $\tilde{\gamma}_{**}>0$ sufficiently small such that
\begin{equation}\label{3-80}
\tilde{s}_{**}\leq s_{**},
\end{equation}
and
\begin{align}\label{3-81}
&\frac{2N^mK_D(e^{2-m}+e)}{C_{10}}\varepsilon^{(1-\varepsilon)(m-1)}(1-\varepsilon)d^{m-2}c_7(\varepsilon)
\nonumber
\\
&\times \tilde{s}_{**}^{-\frac{2}{N}+2\sigma(1-\frac{1}{N})+(\varepsilon-1)(m-2)}
\leq\frac{1}{2},
\end{align}
as well as
\begin{equation}\label{3-82}
\frac{2e(1-\varepsilon)}{NC_{10}d}\tilde{\gamma}_{**}
\leq\frac{1}{2},
\end{equation}
with
\begin{equation}\label{3-83}
c_7(\varepsilon):=\max
\begin{Bmatrix}
1,\varepsilon^{(\varepsilon-1)(m-2)+\varepsilon-2}
\end{Bmatrix}.
\end{equation}
Thanks to \eqref{1-8},  \eqref{P2}, \eqref{3-56}, \eqref{3-76}, \eqref{3-80}, $\widehat{U}_{ss}\leq 0$ by \eqref{3-28}, and the nonnegativity of $\widehat{U}$, we know that
\begin{align}\label{3-84}
&\mathcal{P}_2[\widehat{U},\widehat{W}](s,t)
\nonumber
\\
=&\widehat{U}_t-\theta\widehat{U}-N^2s^{2-\frac{2}{N}}D(Ne^{-\theta t}\widehat{U}_s)\widehat{U}_{ss}
-\frac{Nd}{2e}C_{10}\widehat{U}_s\cdot
s^{\varepsilon+2\sigma(\frac{1}{N}-1)}
\nonumber
\\
\leq&\widehat{U}_t
-N^2s^{2-\frac{2}{N}}K_D(Ne^{-\theta t}\widehat{U}_s)^{m-1}\widehat{U}_{ss}
-\frac{Nd}{2e}C_{10}\widehat{U}_s\cdot
s^{\varepsilon+2\sigma(\frac{1}{N}-1)},
\end{align}
for all $t\in(0,T)$ and $s\in(\frac{1}{g(t)},R^N)$.
Employing \eqref{3-27} and \eqref{3-28}, we conclude that
\begin{align}\label{3-85}
&\frac{-N^2s^{2-\frac{2}{N}}K_D(Ne^{-\theta t}\widehat{U}_s)^{m-1}\widehat{U}_{ss}}
{\frac{Nd}{2e}C_{10}\widehat{U}_s\cdot
s^{\varepsilon+2\sigma(\frac{1}{N}-1)}}
\nonumber
\\
=&-\frac{2N^mK_De^{1-(m-1)\theta t}}{dC_{10}}\cdot s^{2-\frac{2}{N}-\varepsilon+2\sigma(1-\frac{1}{N})}\widehat{U}_s^{m-2}\widehat{U}_{ss}
\nonumber
\\
=&\frac{2N^mK_De^{1-(m-1)\theta t}}{C_{10}}\varepsilon^{(1-\varepsilon)(m-1)}(1-\varepsilon)d^{m-2}
\nonumber
\\
&\times s^{2-\frac{2}{N}-\varepsilon+2\sigma(1-\frac{1}{N})}\left(s-\frac{1-\varepsilon}{g(t)}\right)^{(\varepsilon-1)(m-2)+\varepsilon-2},
\end{align}
for all $t\in(0,T)$ and $s\in(\frac{1}{g(t)},R^N)$.
Recalling \eqref{3-55}, \eqref{3-78} and \eqref{3-83},  we infer that
\begin{align}\label{3-86}
&s^{2-\frac{2}{N}-\varepsilon+2\sigma(1-\frac{1}{N})}\left(s-\frac{1-\varepsilon}{g(t)}\right)^{(\varepsilon-1)(m-2)+\varepsilon-2}
\nonumber
\\
\leq&c_7(\varepsilon)\cdot s^{-\frac{2}{N}+2\sigma(1-\frac{1}{N})+(\varepsilon-1)(m-2)}
\nonumber
\\
\leq&c_7(\varepsilon)\cdot \tilde{s}_{**}^{-\frac{2}{N}+2\sigma(1-\frac{1}{N})+(\varepsilon-1)(m-2)},
\end{align}
for all $\varepsilon\in(0,\varepsilon_{***})$, and $t\in(0,T)$, $s\in(\frac{1}{g(t)},R^N)\cap(0,\tilde{s}_{**})$.
Because of \eqref{3-60}, \eqref{3-81}, \eqref{3-85} and \eqref{3-86}, we obtain that
\begin{equation}\label{3-87}
-N^2s^{2-\frac{2}{N}}K_D(Ne^{-\theta t}\widehat{U}_s)^{m-1}\widehat{U}_{ss}
\leq\frac{1}{2}\cdot \frac{Nd}{2e}C_{10}\widehat{U}_s\cdot
s^{\varepsilon+2\sigma(\frac{1}{N}-1)},
\end{equation}
for all $t\in(0,T)\cap(0,\frac{1}{\theta})$ and $s\in(\frac{1}{g(t)},R^N)\cap(0,\tilde{s}_{**})$.
In view of \eqref{3-26}, \eqref{3-27}, \eqref{3-75} and \eqref{3-79}, we conclude that
\begin{align}\label{3-88}
\frac{\widehat{U}_t}
{\frac{Nd}{2e}C_{10}\widehat{U}_s\cdot
s^{\varepsilon+2\sigma(\frac{1}{N}-1)}}
=&\frac{2e(1-\varepsilon)}{NC_{10}d}
\cdot
s^{-\varepsilon+2\sigma(1-\frac{1}{N})}
\cdot
\frac{g'(t)}{g^2(t)}
\nonumber
\\
\leq&\frac{2e(1-\varepsilon)}{NC_{10}d}
\cdot g^{\varepsilon-2\sigma(1-\frac{1}{N})-2}(t)g'(t)
\nonumber
\\
=&\frac{2e(1-\varepsilon)}{NC_{10}d}
\cdot g^{-1-\delta_2}(t)g'(t)
\nonumber
\\
\leq&\frac{2e(1-\varepsilon)}{NC_{10}d}\tilde{\gamma}_{**},
\end{align}
for all $\varepsilon\in(0,\varepsilon_{***})$ and $t\in(0,T)$, $s\in(\frac{1}{g(t)},R^N)$.
\eqref{3-82} and \eqref{3-88} imply
\begin{equation}\label{3-89}
\widehat{U}_t
\leq\frac{1}{2}\cdot \frac{Nd}{2e}C_{10}\widehat{U}_s\cdot
s^{\varepsilon+2\sigma(\frac{1}{N}-1)},
\end{equation}
for all $\varepsilon\in(0,\varepsilon_{***})$ and $t\in(0,T)\cap(0,\frac{1}{\theta})$, $s\in(\frac{1}{g(t)},R^N)\cap(0,\tilde{s}_{**})$.
\eqref{3-34}, \eqref{3-84}, \eqref{3-87} and \eqref{3-89} lead to \eqref{3-77}.
\end{proof}

\subsubsection{$\mathcal{P}[\underline{U},\underline{W}](s,t)\leq 0$ for $\sigma\in\mathbb{R}$ in outer interval.}

\begin{lemma}\label{le3.11}
Let $N\geq 1$, $R>0,~\mu_*,~\mu^*>0$, $\varepsilon\in(0,\frac{1}{4})\cap(0,\frac{1}{N})$ and $\sigma\in\mathbb{R}$. Suppose that $D\in C^3([0,\infty))$ satisfies \eqref{D}. Let $s_0\in(0,1]$ denote an arbitrarily given constant. It follows that there exists $\theta_*>0$, for which the following holds:\\
For any $\theta>\theta_*$, if $T>0$ and $g\in C^1([0,T))$ satisfy
\begin{equation}\label{3-90}
g(t)>\frac{1}{R^N}
\quad\text{and}\quad
0\leq g'(t)\leq g^2(t),
\quad\text{for all}~t\in(0,T),
\end{equation}
then the functions $\underline{U}$ and $\underline{W}$ defined in \eqref{3-33} satisfy
\begin{equation}\label{3-91}
\mathcal{P}[\underline{U},\underline{W}](s,t)\leq0,
\end{equation}
for all $t\in(0,T)\cap(0,\frac{1}{\theta})$ and $s\in(\frac{1}{g(t)},R^N)\cap[s_0,R^N)$.
\end{lemma}

\begin{proof}
The same method as that for \citep[Lemma 3.10]{TW} can be applied to prove the lemma.
\end{proof}

\subsection{Subsolution properties: $\mathcal{Q}[\underline{U},\underline{W}](s,t)\leq 0$.}

\begin{lemma}\label{le3.12}
Let $N\geq1$, $R>0$, $\varepsilon\in(0,\frac{1}{4})$ and $T>0$.
Then for $\underline{U}$ and $\underline{W}$ defined in \eqref{3-33}, and for all
\begin{equation}\label{3-92}
\theta\geq1,
\end{equation}
we conclude that
\begin{equation}\label{3-93}
\mathcal{Q}[\underline{U},\underline{W}](s,t)\leq0,
\quad\text{for all~}
t\in(0,T)\text{ and }s\in(0,R^N).
\end{equation}
\end{lemma}

\begin{proof}
Employ \eqref{3-8}, \eqref{3-30}, \eqref{3-33}, \eqref{3-92}, and the nonnegativity of $\widehat{U}$ and $\widehat{W}$, we yield that
\begin{align*}
e^{\theta t}\cdot\mathcal{Q}[\underline{U},\underline{W}](s,t)
=\widehat{W}_t
-\theta\widehat{W}
+\widehat{W}
-\widehat{U}
\leq-\widehat{U}
\leq0,
\end{align*}
for all $t\in(0,T)$ and $s\in(0,R^N)$.
\end{proof}

\subsection{Proof of Theorem \ref{th1.1}.}

\begin{lemma}\label{le3.13}
Let $N\geq 1$ and $R>0$.~Suppose that $D\in C^3([0,\infty))$ satisfy \eqref{D}, \eqref{1-8} with parameters $\xi_0>0,~K_D>0$, and exponents $m\in\mathbb{R}$ and $\sigma\in\mathbb{R}$ fulfilling
\begin{equation}\label{3-94}
m<2-\frac{2}{N},
\end{equation}
and
\begin{equation}\label{3-95}
\sigma\in
\begin{cases}
\left(-\infty,+\infty\right),
&\text{if}~N=1,
\\
\left(\max\{\frac{N}{2-2N},\frac{mN+2-2N}{2N-2}\},+\infty\right),
&\text{if}~N\geq 2.
\end{cases}
\end{equation}
It follows that there exists $\varepsilon\in(0,\frac{1}{4})\cap(0,\frac{1}{N})$ with the property that for arbitrary $\mu_*,~\mu^*>0$, and $T_*>0$, we can select $\theta>0$, $T\in(0,T_*)$ and a positive function $g\in C^1([0,T))$ such that
\begin{equation}\label{3-96}
g(t)\to+\infty,
\quad\text{as}~t\nearrow T,
\end{equation}
and that the functions $\underline{U}$ and $\underline{W}$ as in \eqref{3-33} satisfy
\begin{equation}\label{3-97}
\mathcal{P}[\underline{U},\underline{W}](s,t)\leq0
\quad\text{and}\quad
\mathcal{Q}[\underline{U},\underline{W}](s,t)\leq0,
\end{equation}
for all $t\in(0,T)$ and $s\in(0,R^N)\setminus\{\frac{1}{y(t)}\}$.
\end{lemma}

\begin{proof}
For $\varepsilon_*$, $\varepsilon_{**}$ and $\varepsilon_{***}$ defined in Lemmas \ref{le3.8}-\ref{le3.10}, one can get any $\varepsilon\in(0,\frac{1}{4})\cap(0,\frac{1}{N})$ such that
\begin{equation*}
\varepsilon<\min\left\{\varepsilon_*,~\varepsilon_{**},~\varepsilon_{***}\right\}.
\end{equation*}
Define
\begin{equation*}
\delta:=\min\left\{\delta_1,~\delta_2,~1-\varepsilon\right\},
\end{equation*}
where $\delta_1$ and $\delta_2$ are shown in Lemma \ref{le3.9} and Lemma \ref{le3.10}.
Given $\mu_*,~\mu^*>0$, $T_*>0$, together with $s_*$, $s_{**}$, $\tilde{s}_{**}$ and $s_0$ as specified in \eqref{3-32}, Lemma \ref{le3.8}, Lemma \ref{le3.10} and Lemma \ref{le3.11}, respectively, we set
\begin{equation*}
s_0=\min\left\{s_*,~s_{**},~\tilde{s}_{**}\right\}.
\end{equation*}
Define
\begin{equation*}
\theta:=\max\left\{\theta_*,~1\right\},
\end{equation*}
where $\theta_*$ is shown in Lemma \ref{le3.11}.
As $d,~\gamma_*,~\gamma_{**},~\tilde{\gamma}_{*},~\tilde{\gamma}_{**}$
in \eqref{d}, Lemmas \ref{le3.7}-\ref{le3.10}, we pick
\begin{equation*}
\gamma:=\min\left\{\gamma_*,~\gamma_{**},~\tilde{\gamma}_{*},~\tilde{\gamma}_{**},~1\right\}.
\end{equation*}
A sufficiently large $g_0>0$ is chosen such that
\begin{equation*}
g_0\geq g_*,
\end{equation*}
where $g_*$ is defined in Lemma \ref{le3.7}, and that
\begin{equation}\label{3-98}
T:=\frac{1}{\gamma\delta g_0^\delta}
\quad\text{satisfies}\quad T<\min\left\{\frac{1}{\theta},~T_*\right\}.
\end{equation}
Given $T>0$ satisfying \eqref{3-98}, the problem
\begin{equation*}
\begin{cases}
g'(t)=\gamma g^{1+\delta}(t),
\quad t\in(0,T),
\\
g(0)=g_0,
&\end{cases}
\end{equation*}
possesses a solution $g\in C^1([0,T))$ which fulfills \eqref{3-96} and $g\geq g_0\geq g_*$. Then all conditions on $g$ specified in Lemmas \ref{le3.7}-\ref{le3.11} are fulfilled. Lemmas \ref{le3.7}-\ref{le3.11} and Remark \ref{remark 3.1} lead to \eqref{3-97} under the condition of \eqref{3-94} and \eqref{3-95}.
It follows from \eqref{3-98} that $(0,T)\cap(0,\frac{1}{\theta})=(0,T)$ and $T<T_*$. Hence, the proof of the lemma is accomplished.
\end{proof}

\begin{proof}[Proof of Theorem \ref{th1.1}.]
According to Lemma \ref{le3.13}, we conclude that $(\underline{U},\underline{W})$ as defined in \eqref{3-33} is a subsolution to \eqref{wen}. For the definition of \eqref{3-10}, we let
\begin{equation*}
M^{(u)}(r)=\omega_N\underline{U}(r^N,0)\quad\text{and}\quad M^{(w)}(r)=\omega_N\underline{W}(r^N,0),
\end{equation*}
for $r\in[0,R]$, where $\omega_N$ denotes the surface area of the unit sphere in $\mathbb{R}^N$.
For $(u_0,w_0)$ satisfying \eqref{1-11}-\eqref{1-13}, combining this with the comparison principle stated in Lemma \ref{le3.2} and \eqref{3-3}, \eqref{3-27}, \eqref{3-33}, \eqref{3-96}, \eqref{3-98}, we derive that
\begin{equation*}
\frac{1}{N}u(0,t)=U_{s}(0,t)\geq\underline{U}_{s}(0,t)=e^{-\theta t}dg^{1-\varepsilon}(t)\geq e^{-1}dg^{1-\varepsilon}(t)\rightarrow\infty,\text{ as }t\nearrow T,
\end{equation*}
which implies that $T_{max}\leq T<\infty$. Then we get the conclusion of Theorem \ref{th1.1}.
\end{proof}

\section{Global existence of solution}
In this section, we establish global existence for solutions to \eqref{wen} in the cases $N=1$ and $N\geq2$.

\subsection{Global existence of solution in $N=1$.}\label{se4.1}

\begin{lemma}\label{le4.1}
Let $\Omega\subset\mathbb{R}$ be a bounded smooth domain. Assume that \eqref{1-16} is valid.
Suppose that $D\in C^2([0,\infty))$ satisfies \eqref{D}, \eqref{1-15} with parameters $\xi_0>0,~k_D>0,~m\in\mathbb{R}~\text{and~}\sigma\in\mathbb{R}$. Then for any
\begin{align}\label{p}
p>p_*:=\max\{3,~m\},
\end{align}
and $T\in(0,T_{\max})$, the following holds:
\begin{itemize}
\item [$(i)$] If $\sigma\in\mathbb{R}$ and $m>0$, then there exists a constant $C_{11}(p)>0$ such that
\begin{equation}\label{4-2}
\left\|u(\cdot,t)\right\|_{L^p}
\leq C_{11}(p),
\quad\text{for all}~t\in(0,T_{\max}).
\end{equation}
\end{itemize}

\begin{itemize}
\item [$(ii)$] If $\sigma\in\mathbb{R}$ and $m=0$, then there exists a constant $C_{12}(p,T)>0$ such that
\begin{equation}\label{4-3}
\left\|u(\cdot,t)\right\|_{L^p}
\leq C_{12}(p,T),
\quad\text{for all}~t\in(0,T).
\end{equation}
\end{itemize}
\end{lemma}

\begin{proof}
Combining the boundary conditions and the second equation of \eqref{wen}, the results in \citep[Remark 1.1]{NT}, we can obtain that
\begin{equation}\label{4-4}
\left\|v(\cdot,t)\right\|_{W^{1,\infty}(\Omega)}
\leq c_8,
\quad\text{for all }t\in(0,T_{\max}).
\end{equation}
It suffices to focus on the region $\left\lbrace(x,t)\in\Omega\times(0,T_{max}):u>\xi_0\right\rbrace $. since the boundedness of $u$ holds naturally on its complement. For convenience, we may assume without loss of generality that $u>\xi_0$ holds throughout $\Omega\times(0,T_{\max})$.
\begin{itemize}
\item [$(i)$]$\sigma\in\mathbb{R}$ and $m>0$.
\end{itemize}
Making use of \eqref{1-15}, \eqref{p}, \eqref{4-4}, Young's inequality, integrating by parts and the first equation of \eqref{wen}, we conclude that
\begin{align*}
&\frac{1}{p}\frac{\mathrm{d}}{\mathrm{d}t}\int_\Omega u^p
+k_D(p-1)\int_\Omega u^{m+p-3}|\nabla u|^2
\nonumber
\\
\leq&\frac{1}{p}\frac{\mathrm{d}}{\mathrm{d}t}\int_\Omega u^p
+(p-1)\int_\Omega u^{p-2}D(u)|\nabla u|^2
\nonumber
\\
=&(p-1)\int_\Omega u^{p-1}(1+\left|\nabla v\right|^2)^{\sigma}\nabla u\cdot\nabla v
\nonumber
\\
\leq&\frac{(p-1)k_D}{2}\int_\Omega u^{m+p-3}|\nabla u|^2
+\frac{p-1}{2k_D}\int_\Omega u^{p-m+1}\left|\nabla v\right|^2\left(1+\left|\nabla v\right|^2\right)^{2\sigma}
\nonumber
\\
\leq&\frac{(p-1)k_D}{2}\int_\Omega u^{m+p-3}|\nabla u|^2
+c_{9}(p)\int_\Omega u^{p-m+1},
\quad\text{for all }t\in(0,T_{\max}),
\end{align*}
which implies
\begin{align}\label{4-5}
&\frac{\mathrm{d}}{\mathrm{d}t}\int_\Omega u^p
+c_{10}(p)\left\|\nabla u^{\frac{p+m-1}{2}}\right\|
_{L^2}^2
+\left(\int_\Omega u^p\right)^{\frac{1}{2}}
\nonumber
\\
\leq&c_{11}(p)\int_\Omega u^{p-m+1}
+\left(\int_\Omega u^p\right)^{\frac{1}{2}},
\quad\text{for all }t\in(0,T_{\max}).
\end{align}
Thanks to the Gagliardo-Nirenberg inequality and \eqref{2-3}, we have that
\begin{align}\label{4-6}
&c_{11}(p)\int_\Omega u^{p-m+1}
\nonumber
\\
=&c_{11}(p)\left\|u^{\frac{p+m-1}{2}}\right\|
_{L^\frac{2(p-m+1)}{p+m-1}}
^{\frac{2(p-m+1)}{p+m-1}}
\nonumber
\\
\leq&c_{12}(p)\left\|\nabla u^{\frac{p+m-1}{2}}\right\|_{L^2}^{\frac{2a_1(p-m+1)}{p+m-1}}
\left\|u^{\frac{p+m-1}{2}}\right\|_{L^\frac{2}{p+m-1}}^{\frac{2(1-a_1)(p-m+1)}{p+m-1}}
+c_{12}(p)\left\|u^{\frac{p+m-1}{2}}\right\|
_{L^\frac{2}{p+m-1}}
^{\frac{2(p-m+1)}{p+m-1}}
\nonumber
\\
\leq&c_{13}(p)\left\|\nabla u^{\frac{p+m-1}{2}}\right\|_{L^2}^{\frac{2a_1(p-m+1)}{p+m-1}}
+c_{14}(p),
\end{align}
and
\begin{align}\label{4-7}
&\left(\int_\Omega u^p\right)^{\frac{1}{2}}
\nonumber
\\
=&\left\|u^{\frac{p+m-1}{2}}\right\|
_{L^\frac{2p}{p+m-1}}
^{\frac{p}{p+m-1}}
\nonumber
\\
\leq&c_{15}(p) \left\|\nabla u^{\frac{p+m-1}{2}}\right\|_{L^2}^{\frac{pa_2}{p+m-1}}
\left\|u^{\frac{p+m-1}{2}}\right\|_{L^\frac{2}{p+m-1}}^{\frac{p(1-a_2)}{p+m-1}}
+c_{15}(p)\left\|u^{\frac{p+m-1}{2}}\right\|
_{L^\frac{2}{p+m-1}}
^{\frac{p}{p+m-1}}
\nonumber
\\
\leq&c_{16}(p)\left\|\nabla u^{\frac{p+m-1}{2}}\right\|_{L^2}^{\frac{pa_2}{p+m-1}}
+c_{17}(p).
\end{align}
Here, for any $p>p_*$,
\begin{equation*}
a_1=\frac{(p-m)(p+m-1)}{(p+m)(p-m+1)}\in(0,1),
\quad
a_2=\frac{(p-1)(p+m-1)}{p(p+m)}\in(0,1),
\end{equation*}
and they satisfy
\begin{equation}\label{4-8}
\frac{2a_1(p-m+1)}{p+m-1}<2,
\quad
\frac{pa_2}{p+m-1}<2,
\quad\text{for any }p>p_*\text{ and }m>0.
\end{equation}
According to \eqref{4-5}-\eqref{4-8} and Young's inequality, we gain that
\begin{equation}\label{4-9}
\frac{\mathrm{d}}{\mathrm{d}t}\int_\Omega u^p
+\left(\int_\Omega u^p\right)^{\frac{1}{2}}
\leq c_{18}(p),
\quad\text{for all }t\in(0,T_{\max}).
\end{equation}
\eqref{4-2} follows from \eqref{4-9} and Lemma \ref{le2.2}.
\begin{itemize}
\item [$(ii)$]$\sigma\in\mathbb{R}$ and $m=0$.
\end{itemize}
Once more using \eqref{1-15}, \eqref{p}, \eqref{4-4}, Young's inequality and integrating by parts on the first equation of \eqref{wen}, we derive that
\begin{align}\label{4-10}
&\frac{1}{p}\frac{\mathrm{d}}{\mathrm{d}t}\int_\Omega u^p
+k_D(p-1)\int_\Omega u^{p-3}|\nabla u|^2
\nonumber
\\
\leq&\frac{1}{p}\frac{\mathrm{d}}{\mathrm{d}t}\int_\Omega u^p
+(p-1)\int_\Omega u^{p-2}D(u)|\nabla u|^2
\nonumber
\\
=&(p-1)\int_\Omega u^{p-1}(1+\left|\nabla v\right|^2)^{\sigma}\nabla u\cdot\nabla v
\nonumber
\\
\leq&c_{19}(p)\left|\int_\Omega u^{p-1}\nabla u\cdot\nabla v\right|
\nonumber
\\
=&c_{20}(p)\left|-\int_\Omega u^p\Delta v\right|
\nonumber
\\
\leq&\varepsilon_1\int_\Omega u^{p+1}
+c_{21}(p,\varepsilon_1)\int_\Omega\left|\Delta v\right|^{p+1},
\quad\text{for all }t\in(0,T_{\max}),
\end{align}
where $\varepsilon_1>0$ is a constant to be chosen later.
Use \eqref{4-10}, the well-known regularity theory for elliptic problems (\cite{WP}) and the second equation of \eqref{wen} to yield that
\begin{align}\label{4-11}
&\frac{1}{p}\frac{\mathrm{d}}{\mathrm{d}t}\int_\Omega u^p
+k_D(p-1)\int_\Omega u^{p-3}|\nabla u|^2
\nonumber
\\
\leq&\varepsilon_1\int_\Omega u^{p+1}
+c_{21}(p,\varepsilon_1)\int_\Omega\left|\Delta v\right|^{p+1}
\nonumber
\\
\leq&\varepsilon_1\int_\Omega u^{p+1}
+c_{22}(p,\varepsilon_1)\int_\Omega w^{p+1},
\quad\text{for all }t\in(0,T_{\max}).
\end{align}
In light of the third equation of \eqref{wen} and Young's inequality, we assert that
\begin{align}\label{4-12}
\frac{1}{p+1}\frac{\mathrm{d}}{\mathrm{d}t}\int_\Omega w^{p+1}
+\int_\Omega w^{p+1}
=&\int_\Omega w^pu
\nonumber
\\
\leq&\varepsilon_1\int_\Omega u^{p+1}
+c_{23}(p,\varepsilon_1)\int_\Omega w^{p+1},
\end{align}
for all $t\in(0,T_{\max})$. Adding \eqref{4-11} and \eqref{4-12} leads to
\begin{align}\label{4-13}
&\frac{\mathrm{d}}{\mathrm{d}t}\int_\Omega\left(u^p+w^{p+1}\right)
+\left\|\nabla u^{\frac{p-1}{2}}\right\|_{L^2}^2
\nonumber
\\
\leq&c_{24}(p)\varepsilon_1\int_\Omega u^{p+1}
+c_{25}(p,\varepsilon_1)\int_\Omega w^{p+1},
\quad\text{for all }t\in(0,T_{\max}).
\end{align}
Applying the Gagliardo-Nirenberg inequality and \eqref{2-3} again, we find that
\begin{align}\label{4-14}
&c_{24}(p)\varepsilon_1\int_\Omega u^{p+1}
\nonumber
\\
=&c_{24}(p)\varepsilon_1\left\|u^{\frac{p-1}{2}}\right\|_{L^{\frac{2(p+1)}{p-1}}}^{\frac{2(p+1)}{p-1}}
\nonumber
\\
\leq&c_{26}(p)\varepsilon_1\left\|\nabla u^{\frac{p-1}{2}}\right\|_{L^2}^{\frac{2(p+1)a_3}{p-1}}
\left\|u^{\frac{p-1}{2}}\right\|_{L^\frac{2}{p-1}}^{\frac{2(p+1)(1-a_3)}{p-1}}
+c_{26}(p)\varepsilon_1\left\|u^{\frac{p-1}{2}}\right\|_{L^\frac{2}{p-1}}^{\frac{2(p+1)}{p-1}}
\nonumber
\\
\leq&c_{27}(p)\varepsilon_1\left\|\nabla u^{\frac{p-1}{2}}\right\|_{L^2}^{\frac{2(p+1)a_3}{p-1}}
+c_{28}(p)\varepsilon_1,
\end{align}
where
\begin{equation*}
a_3=\frac{p-1}{p+1}\in(0,1)
\quad\text{and}\quad
\frac{2(p+1)a_3}{p-1}=2,
\quad\text{for any } p>p_*.
\end{equation*}
Employing \eqref{4-13} and \eqref{4-14}, and letting $\varepsilon_1$ be sufficiently small, we estimate that
\begin{align*}
\frac{\mathrm{d}}{\mathrm{d}t}\int_\Omega\left(u^p+w^{p+1}\right)
\leq c_{29}(p)\int_\Omega w^{p+1}+c_{30}(p),
\quad\text{for all }t\in(0,T_{\max}),
\end{align*}
Then we arrive at \eqref{4-3}.
\end{proof}

\begin{lemma}\label{le4.2}
Let $\Omega\subset\mathbb{R}$ be a bounded smooth domain. Assume that \eqref{1-16} is valid and that $D\in C^2([0,\infty))$ satisfying \eqref{D}, \eqref{1-15} with parameters $\xi_0>0,~k_D>0,~m\in\mathbb{R}~\text{and~}\sigma\in\mathbb{R}$. Then for any $T\in(0,T_{\max})$,  we have the following results:
\begin{itemize}
\item [$(i)$] If $\sigma\in\mathbb{R}$ and $m>0$, then there exists a constant $C_{13}(p)>0$ such that
\begin{equation}\label{4-15}
\left\|u(\cdot,t)\right\|_{L^{\infty}}
\leq C_{13}(p),
\quad\text{for all}~t\in(0,T_{\max}).
\end{equation}
\end{itemize}

\begin{itemize}
\item [$(ii)$] If $\sigma\in\mathbb{R}$ and $m=0$, then there exists a constant $C_{14}(T,p)>0$ such that
\begin{equation}\label{4-16}
\left\|u(\cdot,t)\right\|_{L^{\infty}}
\leq C_{14}(T,p),
\quad\text{for all}~t\in(0,T).
\end{equation}
\end{itemize}
\end{lemma}

\begin{proof}
By applying the Moser-type iteration technique, we are able to derive the results. For a comprehensive derivation process, the reader is referred to \cite{TW2012}.
\end{proof}

\subsection{Global existence of solution in $N\geq2$.}\label{se4.2}

\begin{lemma}\label{le4.3}
Let $\Omega\subset\mathbb{R}^N~(N\geq 2)$ be a bounded smooth domain and let \eqref{1-16} hold. Assume that $D\in C^2([0,\infty))$ satisfies \eqref{D}, \eqref{1-15} with parameters $\xi_0>0,~k_D>0,~m\in\mathbb{R}~\text{and~}\sigma\in\mathbb{R}$. Define
\begin{align}\label{4-17}
\gamma_*:=\max\Biggl\{3,~N+m-1\Biggr\}.
\end{align}
Then for any $T\in(0,T_{\max})$ and sufficiently large $\gamma>\gamma_*$, we obtain the following conclusion:
\begin{itemize}
\item [$(i)$] If $m>1-\frac{1}{N}$ and $\sigma<\frac{mN+2-2N}{2N-2}$, then there exists a constant $C_{15}(\gamma)>0$ such that
\begin{equation}\label{4-18}
\left\|u(\cdot,t)\right\|_{L^\gamma}
\leq C_{15}(\gamma),
\quad\text{for all}~t\in(0,T_{\max}).
\end{equation}
\end{itemize}
\begin{itemize}
\item [$(ii)$] If $m>1-\frac{1}{N}$ and $\sigma\leq\frac{mN+2-2N}{2N-2}$, then there exists a constant $C_{16}(\gamma,T)>0$ such that
\begin{equation}\label{4-19}
\left\|u(\cdot,t)\right\|_{L^\gamma}
\leq C_{16}(\gamma,T),
\quad\text{for all}~t\in(0,T).
\end{equation}
\end{itemize}
\end{lemma}

\begin{proof}
Motivated by \cite{ZY}, for any $m>1-\frac{1}{N}$, the proof is divided into two cases: $\sigma\in(-\infty,-\frac{1}{2}]$ and $\sigma\in(-\frac{1}{2},\frac{mN+2-2N}{2N-2}]$. We may assume without loss of generality that $u>\xi_0$. Otherwise, the boundedness of $u$ would have been already established.
\begin{itemize}
\item~$m>1-\frac{1}{N}$ and $\sigma\leq-\frac{1}{2}$.
\end{itemize}
By means of \eqref{1-15}, \eqref{4-17}, Young's inequality, integrating by parts and the first equation of \eqref{wen}, we infer that
\begin{align*}
&\frac{1}{\gamma}\frac{\mathrm{d}}{\mathrm{d}t}\int_\Omega u^\gamma
+k_D(\gamma-1)\int_\Omega u^{m+\gamma-3}|\nabla u|^2
\nonumber
\\
\leq&\frac{1}{\gamma}\frac{\mathrm{d}}{\mathrm{d}t}\int_\Omega u^\gamma
+(\gamma-1)\int_\Omega u^{\gamma-2}D(u)|\nabla u|^2
\nonumber
\\
=&(\gamma-1)\int_\Omega u^{\gamma-1}(1+\left|\nabla v\right|^2)^\sigma\nabla u\cdot\nabla v
\nonumber
\\
\leq&\frac{k_D(\gamma-1)}{2}\int_\Omega u^{m+\gamma-3}|\nabla u|^2
+\frac{\gamma-1}{2k_D}\int_\Omega u^{\gamma-m+1}\left(1+\left|\nabla v\right|^2\right)^{1+2\sigma}
\nonumber
\\
\leq&\frac{k_D(\gamma-1)}{2}\int_\Omega u^{m+\gamma-3}|\nabla u|^2
+\frac{\gamma-1}{2k_D}\int_\Omega u^{\gamma-m+1},
\end{align*}
for all $t\in(0,T_{\max})$ and $\sigma\leq-\frac{1}{2}$.
Then using $m>1-\frac{1}{N}$ and Young's inequality, we achieve that
\begin{align}\label{4-20}
&\frac{\mathrm{d}}{\mathrm{d}t}\int_\Omega u^\gamma
+\varepsilon_2\left(\int_\Omega u^\gamma\right)^{\frac{1}{2}}
+c_{31}(\gamma)\left\|\nabla u^{\frac{\gamma+m-1}{2}}\right\|_{L^2}^2
\nonumber
\\
\leq&\varepsilon_2\int_\Omega u^{\gamma+m-1+\frac{2}{N}}
+\varepsilon_2\left(\int_\Omega u^\gamma\right)^{\frac{1}{2}}
+c_{32}(\gamma,\varepsilon_2),
\end{align}
for all $t\in(0,T_{\max})$, where $\varepsilon_2>0$ is a constant to be chosen later. Employing the Gagliardo-Nirenberg inequality and \eqref{2-3}, we gain that
\begin{align}\label{4-21}
\varepsilon_2\int_\Omega u^{\gamma+m-1+\frac{2}{N}}
=&\varepsilon_2\left\|u^{\frac{\gamma+m-1}{2}}\right\|_{L^{\frac{2(\gamma+m-1+\frac{2}{N})}{\gamma+m-1}}}^{\frac{2(\gamma+m-1+\frac{2}{N})}{\gamma+m-1}}
\nonumber
\\
\leq&c_{33}(\gamma)\varepsilon_2\left\|\nabla u^{\frac{\gamma+m-1}{2}}\right\|_{L^2}^{\frac{2(\gamma+m-1+\frac{2}{N})a_4}{\gamma+m-1}}
\left\|u^{\frac{\gamma+m-1}{2}}\right\|_{\frac{2}{\gamma+m-1}}^{\frac{2(\gamma+m-1+\frac{2}{N})(1-a_4)}{\gamma+m-1}}
\nonumber
\\
&+c_{33}(\gamma)\varepsilon_2\left\|u^{\frac{\gamma+m-1}{2}}\right\|_{\frac{2}{\gamma+m-1}}^{\frac{2(\gamma+m-1+\frac{2}{N})}{\gamma+m-1}}
\nonumber
\\
\leq&c_{34}(\gamma)\varepsilon_2\left\|\nabla u^{\frac{\gamma+m-1}{2}}\right\|_{L^2}^{\frac{2(\gamma+m-1+\frac{2}{N})a_4}{\gamma+m-1}}
+c_{35}(\gamma)\varepsilon_2,
\end{align}
and
\begin{align}\label{4-22}
&\varepsilon_2\left(\int_\Omega u^\gamma\right)^{\frac{1}{2}}
\nonumber
\\
=&\varepsilon_2\left\|u^{\frac{\gamma+m-1}{2}}\right\|
_{L^\frac{2\gamma}{\gamma+m-1}}
^{\frac{\gamma}{\gamma+m-1}}
\nonumber
\\
\leq&c_{36}(\gamma)\varepsilon_2 \left\|\nabla u^{\frac{\gamma+m-1}{2}}\right\|_{L^2}^{\frac{\gamma a_5}{\gamma+m-1}}
\left\|u^{\frac{\gamma+m-1}{2}}\right\|_{L^\frac{2}{\gamma+m-1}}^{\frac{\gamma(1-a_5)}{\gamma+m-1}}
+c_{36}(\gamma)\varepsilon_2\left\|u^{\frac{\gamma+m-1}{2}}\right\|
_{L^\frac{2}{\gamma+m-1}}
^{\frac{\gamma}{\gamma+m-1}}
\nonumber
\\
\leq&c_{37}(\gamma)\varepsilon_2 \left\|\nabla u^{\frac{\gamma+m-1}{2}}\right\|_{L^2}^{\frac{\gamma a_5}{\gamma+m-1}}
+c_{38}(\gamma)\varepsilon_2.
\end{align}
For any $\gamma>\gamma_*$,
\begin{equation*}
a_4=\frac{\gamma+m-1}{\gamma+m-1+\frac{2}{N}}\in(0,1),
\quad
a_5=\frac{(\gamma-1)(\gamma+m-1)}{\gamma(\gamma+m-2+\frac{2}{N})}\in(0,1),
\end{equation*}
and they satisfy
\begin{equation}\label{4-23}
\frac{2(\gamma+m-1+\frac{2}{N})a_4}{\gamma+m-1}=2,
\quad
\frac{\gamma a_5}{\gamma+m-1}<2,
\end{equation}
for any $\gamma>\gamma_*$ and $m>1-\frac{1}{N}$.
Use \eqref{4-20}-\eqref{4-23}, and choose an appropriately small $\varepsilon_2$ to arrive at
\begin{align*}
\frac{\mathrm{d}}{\mathrm{d}t}\int_\Omega u^\gamma
+c_{39}\left(\int_\Omega u^\gamma\right)^{\frac{1}{2}}
\leq c_{40}(\gamma),
\quad\text{for all }t\in(0,T_{\max}).
\end{align*}
Invoking Lemma \ref{le2.2}, we arrive at
\begin{equation}\label{4-24}
\int_\Omega u^\gamma\leq c_{41}(\gamma),
\quad
\text{for all }t\in(0,T_{\max}).
\end{equation}

\begin{itemize}
\item~$m>1-\frac{1}{N}$ and $-\frac{1}{2}<\sigma\leq\frac{mN+2-2N}{2N-2}$.
\end{itemize}
Applying \eqref{1-15}, \eqref{4-17}, Young's inequality, integrating by parts and the first equation of \eqref{wen}, we see that
\begin{align*}
&\frac{1}{\gamma}\frac{\mathrm{d}}{\mathrm{d}t}\int_\Omega u^\gamma
+k_D(\gamma-1)\int_\Omega u^{m+\gamma-3}|\nabla u|^2
\nonumber
\\
\leq&\frac{1}{\gamma}\frac{\mathrm{d}}{\mathrm{d}t}\int_\Omega u^\gamma
+(\gamma-1)\int_\Omega u^{\gamma-2}D(u)|\nabla u|^2
\nonumber
\\
=&(\gamma-1)\int_\Omega u^{\gamma-1}(1+\left|\nabla v\right|^2)^\sigma\nabla u\cdot\nabla v
\nonumber
\\
\leq&\frac{k_D(\gamma-1)}{2}\int_\Omega u^{m+\gamma-3}|\nabla u|^2
+\frac{\gamma-1}{2k_D}\int_\Omega u^{\gamma-m+1}\left(1+\left|\nabla v\right|^2\right)^{1+2\sigma}
\nonumber
\\
\leq&\frac{k_D(\gamma-1)}{2}\int_\Omega u^{m+\gamma-3}|\nabla u|^2
+c_{42}(\gamma)\int_\Omega u^{\gamma-m+1}\left(1+\left|\nabla v\right|^{2(1+2\sigma)}\right),
\end{align*}
for all $t\in(0,T_{\max})$ and $\sigma>-\frac{1}{2}$. Relying on Young's inequality, we deduce that
\begin{align}\label{4-25}
&\frac{\mathrm{d}}{\mathrm{d}t}\int_\Omega u^\gamma
+c_{43}(\gamma)\left\|\nabla u^{\frac{\gamma+m-1}{2}}\right\|_{L^2}^2
\nonumber
\\
\leq&\varepsilon_3\int_\Omega u^{(\gamma-m+1)q}
+\frac{c_{44}(\gamma)}{\varepsilon_3}\int_\Omega\left|\nabla v\right|^{\frac{2q(1+2\sigma)}{q-1}}
+c_{45}(\gamma,\varepsilon_3),
\end{align}
for all $t\in(0,T_{\max})$, where $\varepsilon_3>0$ denotes a constant that will be chosen later, and where
\begin{equation}\label{q}
q=1+\frac{2(1+2\sigma)(1-\frac{1}{N})}{\gamma-m+1}>1,
\quad\text{for any }\gamma>\gamma_*\text{ and }\sigma>-\frac{1}{2}.
\end{equation}
By applying the regularity theory for elliptic problems (\cite{WP}) again and using the embedding theorem, the second equation of \eqref{wen}, we are able to derive that
\begin{align}\label{4-27}
\frac{c_{44}(\gamma)}{\varepsilon_3}\int_\Omega\left|\nabla v\right|^{\frac{2q(1+2\sigma)}{q-1}}
\leq& \frac{c_{46}(\gamma)}{\varepsilon_3}\left\|v\right\|_{W^{2,\frac{2Nq(1+2\sigma)}{(N+2+4\sigma)q-N}}}^{\frac{2q(1+2\sigma)}{q-1}}
\nonumber
\\
\leq& \frac{c_{47}(\gamma)}{\varepsilon_3}\left\|w\right\|_{L^{\frac{2Nq(1+2\sigma)}{(N+2+4\sigma)q-N}}}^{\frac{2q(1+2\sigma)}{q-1}}.
\end{align}
Making use of \eqref{2-3} and \eqref{3-19}, we assert that
\begin{equation}\label{4-28}
\left\|w(\cdot,t)\right\|_{L^1}\leq c_{48},
\quad\text{for all }t\in(0,T_{\max}).
\end{equation}
For the definition of $q$ in \eqref{q} and any sufficiently large $\gamma>\gamma_*$, we know that
\begin{equation*}
1<\frac{2Nq(1+2\sigma)}{(N+2+4\sigma)q-N}<(\gamma-m+1)q,
\quad\text{for any }\sigma>-\frac{1}{2}.
\end{equation*}
In light of interpolation inequality and \eqref{4-28}, we yield that
\begin{align}\label{4-29}
\frac{c_{47}(\gamma)}{\varepsilon_3}\left\|w\right\|_{L^{\frac{2Nq(1+2\sigma)}{(N+2+4\sigma)q-N}}}^{\frac{2q(1+2\sigma)}{q-1}}
\leq& \frac{c_{47}(\gamma)}{\varepsilon_3}\left\|w\right\|_{L^1}^{\frac{2qb(1+2\sigma)}{q-1}}
\left\|w\right\|_{L^{(\gamma-m+1)q}}^{\frac{2q(1-b)(1+2\sigma)}{q-1}}
\nonumber
\\
\leq&\frac{c_{49}(\gamma)}{\varepsilon_3}\left\|w\right\|_{L^{(\gamma-m+1)q}}^{\frac{2q(1-b)(1+2\sigma)}{q-1}},
\end{align}
where
\begin{equation}\label{4-30}
b=\frac{1}{N},
\end{equation}
by the definition of \eqref{q}.
According to \eqref{4-25}-\eqref{4-27}, \eqref{4-29} and \eqref{4-30}, we conclude that
\begin{align}\label{4-31}
&\frac{\mathrm{d}}{\mathrm{d}t}\int_\Omega u^\gamma
+c_{50}(\gamma)\left\|\nabla u^{\frac{\gamma+m-1}{2}}\right\|_{L^2}^2
\nonumber
\\
\leq&\varepsilon_3\int_\Omega u^{(\gamma-m+1)q}
+\frac{c_{49}(\gamma)}{\varepsilon_3}\left\|w\right\|_{L^{(\gamma-m+1)q}}^{\frac{2q(1-b)(1+2\sigma)}{q-1}}
+c_{45}(\gamma,\varepsilon_3)
\nonumber
\\
=&\varepsilon_3\int_\Omega u^{\gamma-m+1+2(1+2\sigma)(1-\frac{1}{N})}
+\frac{c_{49}(\gamma)}{\varepsilon_3}\left\|w\right\|_{L^{\gamma-m+1+2(1+2\sigma)(1-\frac{1}{N})}}^{\gamma-m+1+2(1+2\sigma)(1-\frac{1}{N})}
\nonumber
\\
&+c_{45}(\gamma,\varepsilon_3),
\quad\text{for all }t\in(0,T_{\max}).
\end{align}
Recalling the third equation of \eqref{wen} and Young's inequality, we establish that
\begin{align}\label{4-32}
&\frac{1}{\gamma-m+1+2(1+2\sigma)(1-\frac{1}{N})}\frac{\mathrm{d}}{\mathrm{d}t}\int_\Omega w^{\gamma-m+1+2(1+2\sigma)(1-\frac{1}{N})}
\nonumber
\\
=&\int_\Omega w^{\gamma-m+2(1+2\sigma)(1-\frac{1}{N})}u
-\int_\Omega w^{\gamma-m+1+2(1+2\sigma)(1-\frac{1}{N})}
\nonumber
\\
\leq&\varepsilon_3\int_\Omega u^{\gamma-m+1+2(1+2\sigma)(1-\frac{1}{N})}
+(\frac{c_{51}(\gamma)}{\varepsilon_3}-1)\int_\Omega w^{\gamma-m+1+2(1+2\sigma)(1-\frac{1}{N})},
\end{align}
for all $t\in(0,T_{\max})$.
Add \eqref{4-31} and \eqref{4-32} to gain that
\begin{align}\label{4-33}
&\frac{\mathrm{d}}{\mathrm{d}t}\int_\Omega\left(u^\gamma+w^{\gamma-m+1+2(1+2\sigma)(1-\frac{1}{N})}\right)
+c_{52}(\gamma)\left\|\nabla u^{\frac{\gamma+m-1}{2}}\right\|_{L^2}^2
\nonumber
\\
\leq&c_{53}(\gamma)\varepsilon_3\int_\Omega u^{\gamma-m+1+2(1+2\sigma)(1-\frac{1}{N})}
\nonumber
\\
&+c_{53}(\gamma)(\frac{c_{51}(\gamma)}{\varepsilon_3}-1)\left\|w\right\|_{L^{\gamma-m+1+2(1+2\sigma)(1-\frac{1}{N})}}^{\gamma-m+1+2(1+2\sigma)(1-\frac{1}{N})}
+c_{54}(\gamma,\varepsilon_3),
\end{align}
for all $t\in(0,T_{\max})$.
Applying the Gagliardo-Nirenberg inequality and \eqref{2-3}, we gain that
\begin{align}\label{4-34}
&c_{53}(\gamma)\varepsilon_3\int_\Omega u^{\gamma-m+1+2(1+2\sigma)(1-\frac{1}{N})}
\nonumber
\\
=&c_{53}(\gamma)\varepsilon_3\left\|u^{\frac{\gamma+m-1}{2}}\right\|_{L^{\frac{2\left(\gamma-m+1+2(1+2\sigma)(1-\frac{1}{N})\right)}{\gamma+m-1}}}^{\frac{2\left(\gamma-m+1+2(1+2\sigma)(1-\frac{1}{N})\right)}{\gamma+m-1}}
\nonumber
\\
\leq&c_{55}(\gamma)\varepsilon_3\left\|\nabla u^{\frac{\gamma+m-1}{2}}\right\|_{L^2}^{\frac{2a_6\left(\gamma-m+1+2(1+2\sigma)(1-\frac{1}{N})\right)}{\gamma+m-1}}
\left\|u^{\frac{\gamma+m-1}{2}}\right\|_{L^\frac{2}{\gamma+m-1}}^{\frac{2(1-a_6)\left(\gamma-m+1+2(1+2\sigma)(1-\frac{1}{N})\right)}{\gamma+m-1}}
\nonumber
\\
&+c_{55}(\gamma)\varepsilon_3\left\|u^{\frac{\gamma+m-1}{2}}\right\|_{L^\frac{2}{\gamma+m-1}}^{\frac{2\left(\gamma-m+1+2(1+2\sigma)(1-\frac{1}{N})\right)}{\gamma+m-1}}
\nonumber
\\
\leq&c_{56}(\gamma)\varepsilon_3\left\|\nabla u^{\frac{\gamma+m-1}{2}}\right\|_{L^2}^{\frac{2a_6\left(\gamma-m+1+2(1+2\sigma)(1-\frac{1}{N})\right)}{\gamma+m-1}}
+c_{56}(\gamma)\varepsilon_3,
\end{align}
where
\begin{equation*}
a_6=\frac{(\gamma+m-1)\left(1-\frac{1}{\gamma-m+1+2(1+2\sigma)(1-\frac{1}{N})}\right)}{\gamma+m-2+\frac{2}{N}}\in(0,1),
\end{equation*}
for any $\sigma>-\frac{1}{2}$ and all sufficiently large $\gamma>\gamma_*$.
Note that $a_6$ satisfies
\begin{equation}\label{4-35}
\frac{2a_6\left(\gamma-m+1+2(1+2\sigma)(1-\frac{1}{N})\right)}{\gamma+m-1}<2,
\quad\text{for }\sigma<\frac{mN+2-2N}{2N-2},
\end{equation}
and
\begin{equation}\label{4-36}
\frac{2a_6\left(\gamma-m+1+2(1+2\sigma)(1-\frac{1}{N})\right)}{\gamma+m-1}=2,
\quad\text{for }\sigma=\frac{mN+2-2N}{2N-2}.
\end{equation}
\
\begin{itemize}
\item~$m>1-\frac{1}{N}$ and $\sigma<\frac{mN+2-2N}{2N-2}$.
\end{itemize}
Due to \eqref{4-33}, we have that
\begin{align*}
&\frac{\mathrm{d}}{\mathrm{d}t}\int_\Omega\left(u^\gamma+w^{\gamma-m+1+2(1+2\sigma)(1-\frac{1}{N})}\right)
+c_{52}(\gamma)\left\|\nabla u^{\frac{\gamma+m-1}{2}}\right\|_{L^2}^2
+\varepsilon_3\left(\int_\Omega u^\gamma\right)^{\frac{1}{2}}
\nonumber
\\
\leq&c_{53}(\gamma)\varepsilon_3\int_\Omega u^{\gamma-m+1+2(1+2\sigma)(1-\frac{1}{N})}
+\varepsilon_3\left(\int_\Omega u^\gamma\right)^{\frac{1}{2}}
\nonumber
\\
&+c_{53}(\gamma)(\frac{c_{51}(\gamma)}{\varepsilon_3}-1)\left\|w\right\|_{L^{\gamma-m+1+2(1+2\sigma)(1-\frac{1}{N})}}^{\gamma-m+1+2(1+2\sigma)(1-\frac{1}{N})}
+c_{54}(\gamma,\varepsilon_3),
\quad\text{for all }t\in(0,T_{\max}).
\end{align*}
Then invoking \eqref{4-22}, \eqref{4-23}, \eqref{4-34}, \eqref{4-35}, Young's inequality and Lemma \ref{le2.2}, as well as letting $\varepsilon_3$ be appropriately large, we estimate that
\begin{equation}\label{4-37}
\int_\Omega\left(u^\gamma+w^{\gamma-m+1+2(1+2\sigma)(1-\frac{1}{N})}\right)
\leq c_{57}(\gamma),
\quad\text{for all }t\in(0,T_{\max}).
\end{equation}

\begin{itemize}
\item~$m>1-\frac{1}{N}$ and $\sigma=\frac{mN+2-2N}{2N-2}$.
\end{itemize}
Employing \eqref{4-33}, \eqref{4-34}, \eqref{4-36} and choosing $\varepsilon_3$ to be sufficiently small, we know that
\begin{align*}
\frac{\mathrm{d}}{\mathrm{d}t}\int_\Omega\left(u^\gamma+w^{\gamma-m+1+2(1+2\sigma)(1-\frac{1}{N})}\right)
\leq c_{58}(\gamma)\left\|w\right\|_{L^{\gamma-m+1+2(1+2\sigma)(1-\frac{1}{N})}}^{\gamma-m+1+2(1+2\sigma)(1-\frac{1}{N})}
+c_{59}(\gamma),
\end{align*}
for all $t\in(0,T_{\max})$.
Then for any $T\in(0,T_{\max})$, we deduce that
\begin{equation}\label{4-38}
\int_\Omega\left(u^\gamma+w^{\gamma-m+1+2(1+2\sigma)(1-\frac{1}{N})}\right)\leq c_{60}(\gamma,T),
\quad\text{for any }t\in(0,T).
\end{equation}
In combination with \eqref{4-24}, \eqref{4-37} and \eqref{4-38}, we have finished the proof.
\end{proof}

Analogous to the approach employed in the proof of Lemma \ref{le4.2}, we are able to derive the subsequent lemma by utilizing the Moser iteration technique.

\begin{lemma}\label{le4.4}
Let $\Omega\subset\mathbb{R}^N~(N\geq 2)$ be a bounded smooth domain and let \eqref{1-16} hold. Assume that $D\in C^2([0,\infty))$ satisfies \eqref{D}, \eqref{1-15} with parameters $\xi_0>0,~k_D>0,~m\in\mathbb{R}~\text{and~}\sigma\in\mathbb{R}$.
Then for any $T\in(0,T_{\max})$, we can get the following results:
\begin{itemize}
\item [$(i)$] If $m>1-\frac{1}{N}$ and $\sigma<\frac{mN+2-2N}{2N-2}$, then there exists a constant $C_{17}>0$ such that
\begin{equation*}
\left\|u(\cdot,t)\right\|_{L^\infty}
\leq C_{17},
\quad\text{for all}~t\in(0,T_{\max}).
\end{equation*}
\end{itemize}
\begin{itemize}
\item [$(ii)$] If $m>1-\frac{1}{N}$ and $\sigma\leq\frac{mN+2-2N}{2N-2}$, then there exists a constant $C_{18}(T)>0$ such that
\begin{equation*}
\left\|u(\cdot,t)\right\|_{L^\infty}
\leq C_{18}(T),
\quad\text{for all}~t\in(0,T).
\end{equation*}
\end{itemize}
\end{lemma}

\subsection{The proof of Theorem \ref{th1.2}.}
According to \eqref{2-2}, Lemma \ref{le4.2} and Lemma \ref{le4.4}, we gain $T_{\max}=\infty$ and the boundedness  of the solution for \eqref{wen} in Theorem \ref{th1.2}.

\end{document}